\documentclass[12pt,english,reqno]{amsart}
\usepackage{amsmath,amssymb,amsthm,amsfonts,graphicx,color}
\usepackage{amssymb,geometry}

\usepackage{color, graphics}

\usepackage{graphicx}
\usepackage[T1]{fontenc}
\usepackage[latin1]{inputenc}
\usepackage[english]{babel}
\usepackage{geometry}

\usepackage{cases}

\usepackage[colorlinks=true]{hyperref}

\usepackage[colorlinks=true]{hyperref}
\hypersetup{linkcolor=black,citecolor=black,filecolor=black,urlcolor=black} 

\newcommand{\x}{{\tt x}}
\newcommand{\y}{{\tt y}}
\newcommand{\tc}{{\tt t}}

\newcommand{\N}{\mathbb{N}}

\newcommand{\R}{\mathbb{R}}

\newcommand{\diver}{\text{div}}

\newcommand{\eps}{\varepsilon}

 \newtheorem{lemma}{Lemma}[section]
\newtheorem{definition}{Definition}[section]
\newtheorem{theorem}{Theorem}[section]
\newtheorem{proposition}{Proposition}[section]
\newtheorem{corollary}{Corollary}[section]
\newtheorem{remark}{Remark}[section]
\newcommand{\bremark}{\begin{remark} \em}
\newcommand{\eremark}{\end{remark} }

\numberwithin{equation}{section}

\title{The Jacobi operator of some special minimal hypersurfaces}

\author{Oscar Agudelo}
\address{University of West Bohemia in Pilsen-NTIS, Univerzitn\'{i} 22, Czech Republic.}
\email {oiagudel@ntis.zcu.cz}

\author{Matteo Rizzi}
\address{Mathematisches Institut, Justus Liebig Universit\"{a}t, Arndtstrasse 2, 35392, Giessen, Germany.}
\email{mrizzi1988@gmail.com}
\thanks{O. Agudelo was supported by the Grant 22-18261S of the Grant Agency of the Czech Republic. M. Rizzi was partially supported by the Alexander von Humboldt foundation.}

\begin{document}
\maketitle

\begin{abstract}
In this work we discuss stability and nondegeneracy properties of some special families of minimal hypersurfaces embedded in $\R^m\times \R^n$ with $m,n\geq 2$. These hypersurfaces are asymptotic at infinity to a fixed Lawson cone $C_{m,n}$. In the case $m+n\ge 8$, we show that such hypersurfaces are strictly stable and we provide a full classification of their bounded Jacobi fields, which in turn allows us to prove the non-degeneracy of such surfaces. In the case $m+n\le 7$, we prove that such hypersurfaces have infinite Morse index.
\end{abstract}


\section{Introduction}
In this work we are interested in stability and nondegeneracy properties of certain minimal hypersurfaces $\Sigma$ embedded in $\R^{N+1}$ with $N\ge 3$. To make our aims precise let us start with some preliminaries.\\

In what follows we assume that $\Sigma\subset \R^{N+1}$ and we will denote its singular set by ${\rm sing}(\Sigma)$. Assume that ${\rm sing}(\Sigma)$ has zero {\it $N-$dimensional Hausdorff measure} and that $\Sigma \setminus {\rm sing}(\Sigma)$ is an orientable hypersurface. Let  $\nu_\Sigma:\Sigma \setminus {\rm sing}(\Sigma) \to S^{N}$ denote a choice of the continuous unit normal vector to $\Sigma$. Here $S^N$ denotes the standard unit sphere in $\R^{N+1}$.\\

These assumptions allow us in particular to consider the Fermi coordinates 
$$
(y,z)\mapsto y +z\nu_{\Sigma}(y)
$$
as a smooth diffeomorphism from $\Sigma \setminus {\rm sing}(\Sigma)\times (-\delta,\delta)$ onto the open set set 
$$\mathcal{N}:=\big\{y + z\nu_{\Sigma}(y)\,:\, y\in \Sigma\setminus {\rm sing}(\Sigma) \quad \hbox{and}\quad |z|<\delta\big\}.
$$

It is known that the area of $\Sigma$ is given by the formula 
$$
{\rm Area}(\Sigma):=\int_\Sigma d\sigma,
$$
where $d\sigma$ is the $N-$dimensional Haussdorff measure. We note that ${\rm Area}(\Sigma)\in [0,\infty]$ and that whenever $\Sigma$ is smooth, ${\rm Area}(\Sigma)$ coincides with the classic definition of area of a manifold with respect to the metric induced by the inner product in $\R^{N+1}$.\\

Minimality, stability and nondegeneracy properties of $\Sigma$ can be understood through appropriate derivatives of ${\rm Area}(\Sigma)$. In order to make these concepts precise we next introduce the notion of normal perturbation of $\Sigma$.\\
Let $\delta>0$ be small and fixed. Consider a function $\phi\in C^\infty_c(\Sigma\setminus {\rm sing}(\Sigma))$, a function $f:\Sigma\times(-\delta,\delta)\to\R^{N+1}$ such that $f$ is smooth in $\big(\Sigma\setminus {\rm sing}(\Sigma)\big)\times (-\delta,\delta)$, $f=f(y,\eps)$ and 
$$
f(y,0)=0 \quad  \mbox{and} \quad \partial_\eps f(y,0)=\phi(y)\nu_\Sigma(y) \quad \mbox{for} \quad y\in \Sigma\setminus {\rm sing}(\Sigma).
$$
For $\eps\in (-\delta,\delta)$ the $\eps-$normal variation of $\Sigma$ with respect to the function $f$ is given by the hypersurface 
\begin{equation}
\label{def-normal-variation}
\Sigma_\eps:=\{y+f(y,\eps):\,y\in\Sigma\} \quad \hbox{for}.
\end{equation}
Let $B\subset \R^{N+1}$ be an open ball such that  ${\rm supp}(\phi)\subset B$ and consider the function 
$$
 (-\delta,\delta)\ni \eps \mapsto {\rm Area}(\Sigma_{\eps}\cap B).
 $$
 
Notice that ${\rm Area}(\Sigma_{\eps}\cap B)\in [0,\infty)$ and it is a smooth function. Even more,
$$
{\rm Area}(\Sigma_{\eps}\cap B)=\int_{\Sigma_{\eps}\cap B}d\sigma.
$$

It is well known that the local  change of area of $\Sigma$ with respect to the formal perturbation $f$ is given by 
\begin{equation}\label{eqn:der_area_sigma_eps}
\frac{d}{d\eps} {\rm Area}(\Sigma_\eps\cap B)\bigg|_{\eps=0}=-\int_\Sigma H_\Sigma \phi d\sigma,
\end{equation}
where $H_{\Sigma}:= \kappa_1 + \cdots + \kappa_N$ is the mean curvature of $\Sigma$ and $\kappa_1,\ldots,\kappa_N: \Sigma \setminus {\rm sing}(\Sigma)\to \R$ correspond to the principal curvatures of $\Sigma$. We remark that the left-hand side in \eqref{eqn:der_area_sigma_eps} does not depend on $B$ but only on $\phi$.

The hypersurface $\Sigma$ is said to be \textit{minimal} if $H_\Sigma\equiv 0$ in $\Sigma\backslash{\rm sing}(\Sigma)$ or equivalently, since the function $\phi$ is arbitrary, if $\Sigma$ is a critical point of the Area functional with respect to compactly supported normal variations. Here it is crucial that $\phi$ is supported outside the singular set of $\Sigma$ and $\Sigma$ is smooth outside ${\rm sing}(\Sigma)$.

Assume that $\Sigma$ is minimal. In the above notations for the normal variation of $\Sigma$, we introduce the second variation of ${\rm Area}(\Sigma)$, which is given by
\begin{equation}
\label{second-variation-Area}
\mathcal{Q}_\Sigma(\phi):=\frac{d^2}{d\eps^2}{\rm Area}(\Sigma_\eps\cap B)\bigg|_{\eps=0}=\int_\Sigma \big(|\nabla_\Sigma \phi|^2-|A_\Sigma|^2\phi^2) \, d\sigma.
\end{equation}
where $B\subset\R^{N+1}$ is any open ball such that ${\rm supp}(\phi)\subset B$. We refer the reader to \cite{Nunes} for the detailed calculations.\\

Now, observe that $\Sigma$ being minimal does not necessarily imply that $\Sigma$ minimises the area functional. In order to  describe the character of the minimality of $\Sigma$, the second variation of ${\rm Area}(\Sigma)$ with respect to the normal variation $\Sigma_{\eps}$ plays a crucial role. Hence we introduce the following definition.
\begin{definition}
\label{def-Sigma-stable}
Let $\Sigma\subset \R^{N+1}$ be a minimal hypersurface of codimension one with ${\rm sing}(\Sigma)$ having measure zero and such that $\Sigma\setminus {\rm sing}(\Sigma)$ is an orientable hypersurface. Assume that $\Sigma$ is minimal. We say that
\begin{enumerate}
\item $\Sigma$ is stable if for any $\phi\in C^\infty_c(\Sigma\setminus  {\rm sing}(\Sigma))$, $\mathcal{Q}_\Sigma(\phi)\ge 0$;  
\item $\Sigma$ is strictly stable if for any $\phi\in C^\infty_c(\Sigma\setminus {\rm sing}(\Sigma))\setminus\{0\}$, $\mathcal{Q}_\Sigma(\phi)>0$.
\end{enumerate}
\end{definition}
It follows from the definitions and Taylor expansions that any minimising hypersurface is stable.\\

Integrating by parts it is possible to see that
$$\mathcal{Q}_\Sigma(\phi)=\int_\Sigma -J_\Sigma \phi\,\phi\, d\sigma\qquad\forall\,\phi\in C^\infty_c(\Sigma\setminus {\rm sing}(\Sigma)),$$
where $J_\Sigma:=\Delta_\Sigma+|A_\Sigma|^2$ is known as the \textit{Jacobi operator} of $\Sigma$, $\Delta_\Sigma$ is the Laplace-Beltrami operator of $\Sigma$ and $|A_\Sigma|^2:=\kappa_1^2+\dots+\kappa_N^2$ is the squared norm of the second fundamental form, or equivalently the sum of the squared principal curvatures of $\Sigma$. A distributional solution $\phi$ to the Jacobi equation 
\begin{equation}
\label{Jacobi-eq}
    J_\Sigma\phi=0
\end{equation}
is known as a \textit{Jacobi field} of $\Sigma$. We are also interested in Jacobi fields which are not necessarily smooth, for instance unbounded Jacobi fields, even on smooth minimal hypersurfaces (see Section \ref{sec-Jacobi}).\\


As mentioned above, a minimal hypersurface is not necessarily a minimiser of the Area functional. For this reason, it is worth introducing the concept of \textit{Area minimising hypersurfaces}.\\

We define the perimeter of a measurable subset $E\subset\R^{N+1}$ in an open set $\Omega\subset\R^{N+1}$ as
\begin{equation}\label{Perimeter}
{\rm Per}(E,\Omega):=\sup\left\{\int_E \diver X \,d\xi:\, X\in C^\infty_c(\Omega,\R^{N+1}),\,|X|\le 1\right\}.
\end{equation}
If $E$ has smooth boundary $\partial E$, it follows from the Divergence Theorem that ${\rm Per}(E,\Omega)$ coincides with the $N$-dimensional {\it Hausdorff measure} of $\partial E\cap\Omega$. However, the definition in \eqref{Perimeter} allows us to treat the case of sets $E$ with non-smooth boundary $\partial E$.\\


\color{black}Assuming, without loss of generality, that $0\in\partial E$, we define area minimising hypersurfaces as follows.
\begin{definition}
\label{def_area-minimising}
We say that $\Sigma:=\partial E$ is an area-minimising (or minimising) hypersurface if for any $\rho>0$ and for any smooth set $F\subset\R^{N+1}$ such that $F\backslash B_\rho(0)=E\backslash B_\rho(0)$,
$$
{\rm Per}(E,B_{2\rho}(0))\le {\rm Per}(F,B_{2\rho}(0)).
$$
\end{definition}
In view of Definition \ref{def_area-minimising}, Area-minimising hypersurfaces are true minimisers of the Area, so that, if they are regular enough, they are minimal hypersurfaces too, while the converse is not necessarily true.\\

A hypersurface $C\subset\R^{N+1}$ is said to be a \textit{cone} if there exists $y_0\in C$ such that 
$$y\in C\text{ if and only if }r (y-y_0)\in C,\,\forall\,r\ge 0.$$
Assuming, up to a translation, that $y_0=0$, any cone can be written as $$C=\{r\theta:\,\theta\in\Gamma,\,r>0\},$$
where $\Gamma=C\cap S^N$ is a suitable $(N-1)$-dimensional hypersurface in the $N$-dimensional unit sphere $S^N$.\\ 

It is known that there are no singular minimising cones of dimension $N\le 6$ (see \cite{A,DeGiorgi,Simons}), while such cones do exist in higher dimension. For instance the Lawson cone
$$C_{m,n}:=\{(x,y)\in\R^m\times\R^n:\,(n-1)|x|^2=(m-1)|y|^2\},\qquad m,n\ge 2$$
is area minimising provided $m+n\ge 9$ or $m+n\ge 8$ and $m,n\ge 3$, while $C_{2,6}$ and $C_{6,2}$ have zero mean curvature everywhere, except at the origin, but they are not area minimising. 
This result was proved in \cite{D} using some previous results by \cite{BDG,La,Sim}.\\

Let $$C:=\{r\theta:\,\theta\in\Gamma,\,r>0\}\subset\R^{N+1}$$ 
be a cone 
with zero mean curvature at any non-singular point and let 
$$\lambda_0<\lambda_1\le\lambda_2\le\dots$$
be the eigenvalues of $-J_\Gamma$ in $H^1(\Gamma)$. It is known that a minimal cone $C$ is stable if and only if
\begin{equation}
\label{strict-stab-cone}
\left(\frac{N-2}{2}\right)^2+\lambda_0\ge 0.
\end{equation}
If (\ref{strict-stab-cone}) is satisfied, we set $\Lambda_0:=\sqrt{\left(\frac{N-2}{2}\right)^2+\lambda_0}$. In particular, $C$ is strictly stable if and only if $\Lambda_0>0$. (See \cite{HS} and references there in). We will see in Section \ref{sec-notations} that $\lambda_0<0$. \\


In this paper we consider minimal hypersurfaces which are asymptotic to some fixed cone $C$. 
\begin{theorem}[\cite{HS}]
\label{th-Sigma}
Let $N\ge 7$. Assume that $C\subset\R^{N+1}$ is a minimising cone. Then there exist exactly two oriented, embedded smooth area minimising hypersurfaces $\Sigma^\pm\subset\R^{N+1}$ such that
\begin{enumerate}
\item $\Sigma^\pm$ are asymptotic to $C$ at infinity and do not intersect $C$.
\item $dist(\Sigma^\pm,\{0\})=1$.
\end{enumerate}
Moreover, the Jacobi field $y\cdotp\nu_{\Sigma^\pm}(y)$ never vanishes.
\end{theorem}
Assumption (\ref{strict-stab-cone}) is fulfilled, for instance, by the Lawson cone $C_{m,n}$ for $m,n\ge 2,\,N+1=m+n\ge 8$, since $\lambda_0=-(N-1)$, hence
$$\Lambda_0=\sqrt{\left(\frac{N-2}{2}\right)^2-(N-1)}>0.$$\\

Such cones enjoy some specific symmetry 
properties. Let us make this statement precise. Given a subgroup $S$ of $O(N+1)$, we say that a hypersurfaces $\Sigma$ is $S$-invariant if for any $y\in\Sigma$ and for any $\rho\in S$, we have $\rho(y) \in\Sigma$. The Lawsons cone are $O(m)\times O(n)$-invariant. Setting
$$E^\pm_{m,n}:=\{(x,y)\in\R^m\times\R^n:\,\pm\left((m-1)|y|^2-(n-1)|x|^2\right)>0\},$$
we have the following result.
\begin{theorem}\label{th_Al}{\cite{AK}}
Let $m,n\ge 2$, $m+n\ge 8$. Then there exist exactly two smooth minimal hypersurfaces $\Sigma^\pm_{m,n}\subset E^\pm_{m,n}$ satisfying that
\begin{enumerate} 
\item $\Sigma^\pm_{m,n}$ are asymptotic to $C_{m,n}$ at infinity.
\item ${\rm dist}(\Sigma^{\pm}_{m,n},\{0\})=1$;
\item  $\Sigma^\pm_{m,n}$ are $O(m)\times O(n)$-invariant.
\end{enumerate}
Moreover, the Jacobi field $y\cdotp\nu_{\Sigma^\pm_{m,n}}(y)$ never vanishes and $\Sigma^\pm_{m,n}$ do not intersect $C_{m,n}$.
\end{theorem}

Theorem \ref{th_Al} is stated in \cite{AK} as a generalisation of Theorem $1.1$ of \cite{AKR}. The proof of Theorem  \ref{th_Al} relies on the constructions in \cite{ABPRS,AK,HS,M,SS}. When either $m,n\ge 3$ with $m+n=8$ or $m,n\ge 2$ with $m+n\ge 9$, the Lawson cone $C_{m,n}$ is area minimising. Thus, in either case Theorem \ref{th_Al} is a particular case of Theorem \ref{th-Sigma}. On the other hand, when either $m=2$ and $n=6$ or $m=6$ and $n=2$, the Lawson cones $C_{2,6}$ and $C_{6,2}$ are strictly stable but not area minimising and in this sense Theorem \ref{th_Al} complements the statement of Theorem \ref{th-Sigma}. We remark that in $\R^{N+1}$ with $N+1\ge 8$, the problem of existence of minimising cones of codimension $1$, which are not Lawson cones is an interesting, but difficult open problem.\\

Each of the hypersurfaces mentioned in Theorems \ref{th-Sigma} and \ref{th_Al} has a positive Jacobi field given by $y\in \Sigma \mapsto |y\cdot\nu_\Sigma(y)|$. As observed in Remark $2.1$ of \cite{AK}, this implies their stability. However, this is not enough to prove their strict stability. In this work we prove the strictly stable character of these hypersurfaces.\\

Next, we discuss a maximum principle and an injectivity result for the Jacobi operator of smooth minimal hypersurfaces which are asymptotic  at infinity to a stable cone. These two results hold in particular for the Jacobi operators of the hypersurfaces constructed in Theorems \ref{th-Sigma} and \ref{th_Al}. Nonetheless, we stress that our results hold for more general hypersurfaces (see Section \ref{sec-EF}).\\

We will see that the strict stability of the asymptotic cone implies the strict stability of the hypersurface, at least in the case of Lawson cones $C_{m,n}$ when $m+n\ge 8$.
\begin{theorem}
\label{th-strict-stability}
The hypersurfaces constructed in Theorem \ref{th_Al} are strictly stable.
\end{theorem}
Then we will consider hypersurfaces which are asymptotic to a fixed Lawson cone in lower dimension, so that condition (\ref{strict-stab-cone}) is not satisfied. 
\begin{theorem}[\cite{ABPRS, M}]
\label{th-Sigma-low-dim}
Let $m,n\ge 2$, $m+n\le 7$. Then there exists a unique complete, embedded, $O(m)\times O(n)$ invariant minimal hypersurface $\Sigma_{m,n}$ such that 
\begin{enumerate}
\item \label{Sigma-as} $\Sigma_{m,n}$ is asymptotic to the cone $C_{m,n}$ at infinity,
\item \label{Sigma-osc} $\Sigma_{m,n}$ intersects $C_{m,n}$ infinitely many times,
\item \label{Sigma-ort} $\Sigma_{m,n}$ meets $\R^m\times\{0\}$ orthogonally,
\item \label{Sigma-dist} $dist(\Sigma_{m,n}, C_{m,n})=1$.
\end{enumerate}
\end{theorem}
For such hypersurfaces it is relevant to compute the Morse index, defined as
\begin{equation}
Morse(\Sigma):=\sup\{\dim(X):\,X\text{ subspace of }C^\infty_c(\Sigma):\,\mathcal{Q}_\Sigma(\phi\backslash sing(\Sigma))<0,\,\forall\,\phi\in X\backslash\{0\}\}.
\end{equation}
It follows from the definition that $Morse(\Sigma)=0$ if and only if $\Sigma$ is stable. For our hypersurfaces we have the following result.
\begin{theorem}
\label{th-Morse-infinite}
The hypersurfaces constructed in Theorem \ref{th-Sigma-low-dim} have infinite Morse index.
\end{theorem}
Theorem \ref{th-Morse-infinite} shows that these hypersurfaces are far from being stable. Roughly speaking, the stability properties of the cone are preserved when passing to the asymptotic smooth minimal hypersurface.\\

Moreover, we are interested in classifying the bounded Jacobi fields of our hypersurfaces, that is bounded solutions to the Jacobi equation (\ref{Jacobi-eq}). Since the zero mean curvature equation is invariant under translations, rotations and dilations, if we consider a normal variation such that, for any $\eps\in(-\delta,\delta)$, $\Sigma_\eps$ is a translated, rotated or dilated copy of $\Sigma$, we have $H_{\Sigma_\eps}=0$ for any $\eps\in(-\delta,\delta)$. Keeping the above notations and differentiating in $\eps$ we have
$$0=\frac{d}{d\eps}\bigg|_{\eps=0}H_{\Sigma_\eps}=J_\Sigma\phi$$
(see \cite{Nunes}), thus there are at most $2N+3$ linearly independent Jacobi fields coming from such geometric transformations. These Jacobi fields and their linear combinations are called \textit{geometric Jacobi fields} and the vector space of such Jacobi fields will be denoted by $G(\Sigma)$. We note that the dimension of $G(\Sigma)$ can be strictly less than $2N+3$ in case $\Sigma$ enjoys some symmetry properties. The Jacobi fields coming from rotations constitute a subspace of dimension at most $N+1$ and are unbounded, so that the space of bounded geometric Jacobi fields has dimension at most $N+2$. In fact, it is generated by the Jacobi fields $\{ \nu_\Sigma(y)\cdotp e_i\}_{1\le i\le N+1}$ associated to translations and by the Jacobi field $y\cdotp\nu_\Sigma(y)$ associated to dilation provided it is bounded.
\begin{definition}
A minimal hypersurface $\Sigma\subset\R^{N+1}$ is said to be nondegenerate if all its bounded Jacobi fields are geometric.
\end{definition}
This definition was given, for instance, in \cite{DKW} for the Costa-Hofmann-Meeks surfaces.\\

Since we are often interested in hypersurfaces enjoying some symmetry properties, we introduce the concept of $S$-nondegenerate hypersurface.
\begin{definition}
Let $S$ be a subgroup of $O(N+1)$ and let $\Sigma\subset\R^{N+1}$ be an $S$-invariant minimal hypersurface. 
\begin{itemize}
    \item We say that a function $\phi:\Sigma\to\R$ is $S$-invariant if $\phi\circ\rho=\phi$, for any $\rho\in S$.
    \item We say that $\Sigma$ is $S$-nondegenerate if all its bounded $S$-invariant Jacobi fields are geometric.
\end{itemize}

\end{definition}
We note that nondegeneracy implies $S$-nondegeneracy, while the converse is not necessarily true.\\

We have nondegeneracy results for our hypersurfaces.
\begin{theorem}
\label{th-nondegeneracy}
\begin{enumerate}
 \item The hypersurfaces constructed in Theorem \ref{th_Al} are nondegenerate.
 \item The hypersurface $\Sigma_{m,n}$ constructed in Theorem \ref{th-Sigma-low-dim} is $O(m)\times O(n)$-nondegenerate, 
 for any $m,n\ge 2$, $m+n\le 7$.
\end{enumerate}
\end{theorem}
We will actually prove more. We will see that the bounded Kernel of the Jacobi operator of $\Sigma$ exactly has dimension $N+2$ if $\Sigma:=\Sigma^\pm_{m,n}$ is one of the hypersurfaces constructed in Theorem \ref{th_Al}. More precisely, there is an $O(m)\times O(n)$-invariant Jacobi field coming from dilation, given by the non-sign changing function $y\mapsto y\cdotp\nu_\Sigma(y)$, and $N+1$ linearly independent Jacobi fields coming from translations which do not enjoy such a symmetry property (see Theorem \ref{th-Jac-Sigma} below). For the hypersurfaces constructed in Theorem \ref{th-Sigma-low-dim}, we prove that the space of bounded $O(m)\times O(n)$-invariant Jacobi fields has dimension $1$ and it is generated by the Jacobi field coming from dilation (see Theorem \ref{th-Jabobi-Sigma-low-dim}).
\begin{remark}
In particular, Theorems \ref{th-strict-stability} and \ref{th-Jac-Sigma} show that strict stability does not exclude the existence of nontrivial bounded and even decaying Jacobi fields. It just rules out the existence of nontrivial too fast decaying Jacobi fields. More precisely, it excludes the existence of Jacobi fields in the space $X$ defined in (\ref{def-X}).
\end{remark}

The plan of the paper is the following. In section \ref{sec-notations} we introduce some basic notions and notations, in section \ref{sec-EF} we prove an injectivity result for the Jacobi operator of a large class of minimal hypersurfaces which are asymptotic to a stable cone at infinity, including the ones mentioned in Theorems \ref{th-Sigma} and \ref{th_Al}, in Section \ref{sec-strict-stability} we prove Theorems \ref{th-strict-stability} and \ref{th-Morse-infinite} about the stability properties of our hypersurfaces and section \ref{sec-Jacobi} will be devoted to nondegeneracy results, that is the proof of Theorem \ref{th-nondegeneracy}. 

\section{Preliminaries and notations}\label{sec-notations}

In this paper we say that a hypersurfaces  $\Sigma\subset\R^{N+1}$ is asymptotic to a cone $$C:=\{r\theta:\,\theta\in\Gamma,\,r>0\}\subset\R^{N+1}$$ 
if, outside a compact set, $\Sigma$ is the normal graph of a smooth decaying function over $C$. More precisely, we introduce the following Definition.
\begin{definition}\label{def-as-cone}
We say that $\Sigma$ is asymptotic to $C$ if there exists $R>0$ such that $$\Sigma=K_R\cup\Sigma_R,$$
where 
\begin{equation}
\label{Sigma-normal-graph}
\Sigma_R:=\{r\theta+w(r,\theta)\nu_C(\theta):\,\theta\in\Gamma,\,r>R\},
\end{equation}
$\nu_C(\theta)=\nu_C(r\theta)$ is a choice of normal vector to $C$, which is actually independent of $r$ due to homogeneity, $w:(R,\infty)\times\Gamma\to \R$ is a smooth function decaying as $r\to\infty$, uniformly in $\theta\in\Gamma$, and $K_R:=\Sigma\backslash \Sigma_R$ is compact. 
\end{definition}
\color{black}In \cite{AA,Simon} the authors give a more abstract definition of hypersurface asymptotic to a cone. However, they actually show that the two definitions are equivalent. See \cite{HS} page 105 and the citations there.\\

In case $C\subset\R^{N+1}$ fulfils (\ref{strict-stab-cone}), or equivalently $C$ is stable, the decay rate of $w$ is discussed in Theorem $4.2$ of \cite{PW}. More precisely, 
we have the following result.
\begin{theorem}[\cite{PW}]
\label{th-dec-w}
Let $C\subset\R^{N+1}$ be a stable cone. Let $\Sigma$ be a minimal hypersurface which, outside a compact set, is the normal graph over $C$ of a smooth function $w:(R,\infty)\times\Gamma\to\R$ which never vanishes and $w(r,\theta)\to 0$ as $r\to\infty$, uniformly in $\theta\in\Gamma$. Then, in the above notations, we have
\begin{equation}
\label{dec-w-slow}
    w(r,\theta)=(b\log r+a)r^{-\frac{N-2}{2}+\Lambda_0}(1+O(r^{-\delta}))
\end{equation}
or
\begin{equation}
\label{dec-w-fast}
w(r,\theta)=ar^{-\frac{N-2}{2}-\Lambda_0}(1+O(r^{-\delta}))
\end{equation}
as $r\to\infty$ uniformly in $\theta\in\Gamma$, for some $\delta>0$, where $a\ne 0$ and $b=0$ unless $\Lambda_0=0$.
These relations can be differentiated.
\end{theorem}
Theorem \ref{th-dec-w} applies to the hypersurfaces constructed in Theorem \ref{th-Sigma} if (\ref{strict-stab-cone}) holds and to the hypersurfaces constructed in Theorem \ref{th_Al}, since the Lawson cone $C_{m,n}$ fulfils (\ref{strict-stab-cone}) for any $m,n\ge 2$, $m+n\ge 8$.
\begin{remark}
\label{rem-dec-w}
    If $\Sigma$ is one of the hypersurfaces constructed in Theorem (\ref{th_Al}), then Theorem (\ref{th-dec-w}) holds with $b=0$, because $$\left(\frac{N-2}{2}\right)^2+\lambda_0=\left(\frac{N-2}{2}\right)^2-(N-1)>0$$
    since $N\ge 7$.
\end{remark}
The next Lemma shows that the result of Theorem \ref{th-dec-w} is consistent with the assumptions, that is the function $w$ is actually decaying at infinity.
\begin{lemma}
If (\ref{strict-stab-cone}) holds, then in the above notations we have $-\frac{N-2}{2}+\Lambda_0<0$.
\end{lemma}
\begin{proof}
It is enough to prove that $\lambda_0<0$. In fact, using the variational characterisation $$\lambda_0=\inf\left\{\int_{\Gamma}-J_\Gamma \phi\,\phi\, d\tilde{\sigma}:\,\phi\in H^1(\Gamma),\,\int_{\Gamma}\phi^2\, d\tilde{\sigma}=1\right\},$$
and taking the constant function $\phi=|\Gamma|^{-1/2}\in H^1(\Gamma)$, where $|\Gamma|$ is the $(N-1)$-dimensional Hausdorff measure of $\Gamma$, so that $\|\phi\|_{L^2(\Gamma)}=1$, we can see that $$\lambda_0\le -\int_\Gamma \frac{|A_\Gamma|^2}{|\Gamma|}d\sigma\le 0.$$
More precisely, we have $\lambda_0<0$, since $|A_\Gamma|^2\ne 0$. In fact, if we assume by contradiction that $|A_\Gamma|^2\equiv 0$, all the principal curvatures $\kappa_i$ of $\Gamma$ are identically zero. This yields that the differential of the Gauss map vanishes identically, so that the normal vector field is constant. This implies that $\Gamma$ is contained in the intersection between $S^N$ and a hyperplane of dimension $N-1$, that is a sphere of dimension $N-2$, which is impossible since $\Gamma$ has dimension $N-1$. 
\end{proof}
The asymptotic behaviour of the functions defining the hypersurfaces constructed in Theorem \ref{th-Sigma-low-dim} is given by the following Theorem, which follows from Propositions $3$ and $5$ of \cite{M} and page $106$ of \cite{HS}.
\begin{theorem}
\label{th-dec-w-low-dim}
Let $m,n\ge 2$ with $N+1=m+n\le 7$. Let $\Sigma:=\Sigma_{m,n}$ be one of the hypersurfaces constructed in Theorem \ref{th-Sigma-low-dim}. Then, outside a compact set, $\Sigma$ is the normal graph over the Lawson cone $C_{m,n}$ of a sign-changing smooth function $w:(R,\infty)\times\Gamma\to\R$ such that $w(r)=O(r^{-\frac{N-2}{2}})$ as $r\to\infty$, uniformly in $\theta\in\Gamma$. 
\end{theorem}

Introducing an Emden-Fowler change of variables $r=e^{\tc}$ on the cone $C$, for $\tc\in\R$, any parametrisation $F_0:\Omega\subset\R^{N-1}\to\Gamma$ of $\Gamma$ induces a parametrisation on $\Sigma$ given by $$F(\tc,\vartheta):=e^\tc F_0(\vartheta)+w(e^\tc,F_0(\vartheta))\nu_C(F_0(\vartheta))\qquad\forall\,(\vartheta,\tc)\in \Omega\times(t_0,\infty),$$ 
which is associated to a metric given by the corresponding first fundamental form $\bar{g}:=(\bar{g}_{ij})_{ij}$, where
$$\bar{g}_{ij}:=\partial_iF\cdotp\partial_jF,\qquad 1\le i,j\le N$$
We note that, In view of Theorems \ref{th-dec-w} and \ref{th-dec-w-low-dim}, there exists $\gamma<0$ such that
\begin{equation}
\label{dec-w}
    w(r,\theta)=o(r^\gamma)\qquad r\to\infty,
\end{equation}
uniformly in $\theta\in\Gamma$ and this relation can be differentiated, so that
\begin{equation}
\label{metric-g-asymptotic}
\bar{g}=e^{2\tc}(1+o(e^{(\gamma-1)\tc}))(d\vartheta^2+d\tc^2), \qquad\text{as $\tc\to\infty$}  
\end{equation} 
where $d\vartheta^2$ is the differential form coming from the metric induced on $\Gamma$ by $F_0$. This holds true for the hypersurfaces mentioned in Theorems \ref{th-dec-w} and \ref{th-Sigma-low-dim}.\\ 

In the sequel, $g$ will denote the extension of $\bar{g}$ to a Riemannian metric on the whole $\Sigma$. Its inverse will be denoted by $g^{ij}:=(g^{-1})_{ij}$. Given a function $u\in C^1(\Sigma)$, its gradient $\nabla_\Sigma u$ is the vector field whose components are
\begin{equation}
\label{def-grad}
u^i:=(\nabla_\Sigma u)^i:=g^{ij}\partial_j u
\end{equation}
so that
\begin{equation}
\label{prod-grad}
\nabla_\Sigma u\cdotp\nabla_\Sigma v=(u^k\partial_k F)\cdotp(u^l\partial_l F)=g^{ij}\partial_i u\partial_j v.
\end{equation}
Given a vector field $v\in C^1(\Sigma,\R^N)$, its divergence is given by
\begin{equation}
    \diver v:=\frac{1}{\sqrt{|g|}}\partial_i(\sqrt{|g|}v^i)
\end{equation}
so that, if $u\in C^2(\Sigma)$, its Laplace-Beltrami operator is given by 
\begin{equation}
\label{def-Lapl-B}
\Delta_\Sigma u:=\diver(\nabla_\Sigma u)=\frac{1}{\sqrt{|g|}}\partial_i(\sqrt{|g|}g^{ij}\partial_j u). 
\end{equation}
We will be particularly interested in $O(m)\times O(n)$-invariant hypersurfaces. This is the case for instance for the hypersurfaces constructed in Theorems \ref{th_Al} and \ref{th-Sigma-low-dim}. Denoting the points by $\xi=(x,y)\in\R^m\times\R^n=\R^{N+1}$, $m,n\ge 2$, and setting $a:=|x|$ and $b:=|y|$, such hypersurfaces can be represented by a curve $(a(s),b(s))$ in the quadrant
$$\{(a,b)\in\R^2:\,a\ge 0,\,b\ge 0\}.$$
In other words, $\Sigma$ can be represented as $$\Sigma=\{(a(s)\x,b(s)\y)\in\R^m\times\R^n:\,s\ge 0,\,\x\in S^{m-1},\,\y\in S^{n-1}\}$$
where $a$ and $b$ satisfy the equation
\begin{equation}\notag
-a''b'+a'b''+(m-1)\frac{b'}{a}+(n-1)\frac{a'}{b}=0,    
\end{equation}
which is equivalent to the zero mean curvature equation $H_\Sigma=0$, and $(a')^2+(b')^2=1$, which guarantees that the curve $(a,b)$ is parametrised by arc-length. The Jacobi equation reduces to
\begin{equation}\notag
\frac{1}{a^2}\Delta_{S^{m-1}}v+\frac{1}{b^2}\Delta_{S^{n-1}}v+\partial_{ss}v+\alpha(s)\partial_sv+\beta(s)v=g,
\end{equation}
where $\phi(y):=v(s,\x,\y)$ and $f(y):=g(s,\x,\y)$, with $y=(a(s)\x,b(s)\y)\in\Sigma$ and
$$\alpha(s):=(m-1)\frac{a'}{a}+(n-1)\frac{b'}{b},\qquad\beta(s)=|A_\Sigma(y)|^2.$$
Introducing an Emden-Fowler change of variables $e^t=s$ and $v(s)=p(t)u(t)$, where 
\begin{equation}
\label{def-p(t)}
p(t):=\exp\left(-\int_0^t \frac{\alpha(e^\tau)e^\tau-1}{2} d\tau\right),
\end{equation}
the corresponding equation for $u$ is given by
\begin{equation}
\label{Jacobi-EF}
\partial_{tt} u+V(t)u=\tilde{g}(t),
\end{equation}
with
\begin{equation}
\label{def-V(t)}
    \tilde{g}(t)=\frac{e^{2t}}{p(t)}g(e^t),\qquad V(t):=\frac{\partial_{tt}p}{p}+(\alpha(e^t)e^t-1)\frac{\partial_{t}p}{p}+e^{2t}\beta(e^t).
\end{equation}
For the complete computations we refer to \cite{AKR,AK}. We stress that $\tc$ denotes the Emden-Fowler variable along the cone $C$ while $t$ denotes the Emden-Fowler variable along $\Sigma$, defined in case $\Sigma$ is $O(m)\times O(n)$-invariant, that is, in our paper, for the hypersurfaces constructed in Theorem \ref{th_Al} and \ref{th-Sigma-low-dim}.

\section{Injectivity result}\label{sec-EF}

In this section $\Sigma$ will be a minimal hypersurface be asymptotic to a cone $C$ at infinity (see Definition \ref{def-as-cone}). Introducing the change of variables $$\phi(y):=|y|^{-\frac{N-2}{2}}u(y)\qquad\forall\,y\in\Sigma$$ 
and defining the operator
$$\mathcal{L}u:=|y|^N\diver(|y|^{2-N}\nabla_\Sigma u)+
\left(|y|^2|A_\Sigma|^2+|y|^{\frac{N+2}{2}}\Delta_\Sigma(|y|^{-\frac{N-2}{2}})\right)u,$$

where $\nabla_\Sigma$ and $\Delta_\Sigma$ are the gradient and the Laplace-Beltrami operator with respect to the metric $g$ defined in Section \ref{sec-notations}, we can prove the following Lemma.
\begin{lemma}
\label{lemma-change-var}
Let $f\in C^0_{loc}(\Sigma)$ and $\phi\in C^2_{loc}(\Sigma)$. Then the Jacobi equation (\ref{Jacobi-eq-rhs}) is satisfied if and only if $u$ and $f$ fulfil
\begin{equation}
\label{new-Jacobi-eq}
\mathcal{L}u=|y|^{\frac{N+2}{2}}f \qquad\text{for a. e. $y\in\Sigma.$}
\end{equation}
\end{lemma}
\begin{remark}
\label{remark-product-metric}
The advantage of the change of variables performed in Lemma \ref{lemma-change-var} is that, due to (\ref{metric-g-asymptotic}), the operator $\mathcal{L}$ is asymptotic to a product operator as $\tc\to\infty$. More precisely, setting
$$u(y)=:\tilde{u}(\tc,\theta)\qquad\forall\, y=e^\tc\theta+w(e^\tc,\theta)\nu_C(\theta)\in\Sigma_R,$$
we have
\begin{equation}
\label{L-as-behaviuor}
\mathcal{L}u(y)=\partial_{\tc\tc}\tilde{u}+\Delta_\Gamma \tilde{u}+\left(|A_\Gamma|^2-\left(\frac{N-2}{2}\right)^2\right)\tilde{u}+\mathcal{R}\tilde{u}=:L\tilde{u}+\mathcal{R}\tilde{u},
\end{equation}
where
\begin{equation}
\label{est-remainder}
|\mathcal{R}\tilde{u}|\le ce^{(\gamma-1)\tc}(|\tilde{u}|+|\partial_\tc\tilde{u}|+|\nabla_\Gamma\tilde{u}|+|\partial_{\tc\tc}\tilde{u}|+|\nabla^2_\Gamma\tilde{u}|+|\partial_\tc\nabla_\Gamma\tilde{u}|)\qquad\forall\,(\tc,\theta)\in(\tc_0,\infty)\times\Gamma
\end{equation}
for any $u\in W^{2,2}_{loc}(\Sigma)$ and  $\tc_0>0$ large enough, since $|y|^2=e^{2\tc}+w^2(e^\tc,\theta)$ and (\ref{dec-w}) is satisfied. This turns out to be very useful to prove the maximum principle outside a compact subset of $\Sigma$ (see Lemma \ref{max-pr-Sigma_tau} below) and to apply the theory in \cite{P}, which leads to prove the existence of a right inverse of the Jacobi operator in some suitable function spaces. We note that $\mathcal{R}=0$ if $\Sigma=C$ is a cone.
\end{remark}
\begin{proof}
Multiplying by $|y|^{-\frac{N+2}{2}}$, equation (\ref{Jacobi-eq-rhs}) is equivalent to
\begin{equation}
\label{Jacobi-eq-mod-y}
|y|^{\frac{N+2}{2}}J_\Sigma(|y|^{-\frac{N-2}{2}}u)=|y|^{\frac{N+2}{2}}f.
\end{equation}
In fact, a computation shows that
\begin{equation}\notag
\begin{aligned}
|y|^{\frac{N+2}{2}} J_\Sigma(|y|^{-\frac{N-2}{2}}u)&=|y|^2(\Delta_\Sigma u+|A_\Sigma|^2u)+
2|y|^{\frac{N+2}{2}}\nabla_\Sigma (|y|^{-\frac{N-2}{2}})\cdotp\nabla_\Sigma u+|y|^{\frac{N+2}{2}}\Delta_\Sigma(|y|^{-\frac{N-2}{2}})u=\\
&=|y|^2(\Delta_\Sigma u+|A_\Sigma|^2u)+(2-N)|y|\nabla_\Sigma|y| \cdotp\nabla_\Sigma u+|y|^{\frac{N+2}{2}}\Delta_\Sigma(|y|^{-\frac{N-2}{2}})u,
\end{aligned}
\end{equation}
On the other hand, it is possible to see that
\begin{equation}
\begin{aligned}
|y|^N\diver(|y|^{2-N}\nabla_{\Sigma} u)
&=\frac{|y|^{N}}{\sqrt{|g|}}\partial_i\left(\frac{\sqrt{|g|}}{|y|^{N}}|y|^2 g^{ij}\partial_j u\right)=\\
&=|y|^2\Delta_\Sigma u+|y|^Ng^{ij}\partial_i((1+|y|)^{2-N})\partial_j u=\\
&=|y|^2\Delta_\Sigma u+(2-N)|y|\nabla_\Sigma |y|\cdotp\nabla_\Sigma u
\end{aligned}
\end{equation}
so that 
\begin{equation}
\label{def-L}
|y|^{\frac{N+2}{2}} J_\Sigma(|y|^{-\frac{N-2}{2}}u)=\mathcal{L}u.
\end{equation}
\end{proof}
First we take $f=0$ and prove injectivity results in suitable spaces, which will be applied to prove strict stability and to classify bounded Jacobi fields.\\

Given a minimal hypersurface $\Sigma\subset\R^{N+1}$ asymptotic to a cone $C$ at infinity, $R>0$ large enough and $\delta\in \R$, we introduce smooth weights $\Gamma_\delta:\Sigma\to(0,\infty)$ defined by $\Gamma_\delta(y):=(1+|y|)^\delta$, which in particular satisfy $\Gamma_{-\delta}=\Gamma_\delta^{-1}$ for any $\delta\in\R$ and $\Gamma_0=1$. We define the weighted spaces
\begin{equation}
L^2(\Sigma,|y|^{-N}):=|y|^{\frac{N}{2}}L^2(\Sigma)\qquad L^2_\delta(\Sigma,|y|^{-N}):=\Gamma_\delta L^2(\Sigma,|y|^{-N}) 
\end{equation}
and
\begin{equation}
\label{def-Sobolev-weighted}
\mathcal{W}^{\ell,2}_\delta(\Sigma,|y|^{-N}):=\{u\in W^{\ell,2}_{loc}(\Sigma):\,\,|y|^j|\nabla^{(j)}_\Sigma u|\in L^2_\delta(\Sigma,|y|^{-N}),\,0\le j\le \ell\},\qquad\ell\in\{1,2\}
\end{equation}
with norms
$$\|u\|_{L^2_\delta(\Sigma,|y|^{-N})}:=\||y|^{-\frac{N}{2}}
\Gamma_{-\delta}u\|_{L^2(\Sigma)}$$ 
and
$$\|u\|_{\mathcal{W}^{\ell,2}_\delta(\Sigma,|y|^{-N})}:=\left(\|u\|^2_{L^2_\delta(\Sigma,|y|^{-N})}+\sum_{j=1}^\ell \||y|^j|\nabla_\Sigma^{(j)}u|\|^2_{L^2_\delta(\Sigma,|y|^{-N})}\right)^{\frac{1}{2}},\qquad\ell\in\{1,2\}.$$
In these notation, the operator $\mathcal{L}$ can be seen as an operator between weighted spaces
\begin{equation}
\label{def-Jacobi-EF}
\begin{aligned}
A_\delta:\,&L^2_\delta(\Sigma,|y|^{-N})\to L^2_\delta(\Sigma,|y|^{-N}),\\
& u\mapsto\mathcal{L}u
\end{aligned}
\end{equation}
with dense domain $\mathcal{W}^{2,2}_\delta(\Sigma,|y|^{-N})$, which will be shown to be injective for some values of $\delta$.\\

In the sequel we will denote the Jacobi field coming from dilations by $\zeta(y):=y\cdotp\nu_\Sigma(y)$.
\begin{remark}
\label{remark-behaviour-Jac-fields}
Let $C\subset\R^{N+1}$ be a stable cone and let $\Sigma$ be a minimal hypersurface asymptotic to $C$ which does not intersect $C$.
\begin{enumerate}
\item If $\Lambda_0>0$, then we have either
\begin{equation}
\label{dec-Jacobi-field-slow}
    \zeta(y)=c|y|^{-\frac{N-2}{2}+\Lambda_0}(1+o(1))\qquad |y|\to\infty.
\end{equation}
or
\begin{equation}
\label{dec-Jacobi-field-fast}
    \zeta(y)=c|y|^{-\frac{N-2}{2}-\Lambda_0}(1+o(1))\qquad |y|\to\infty.
\end{equation}
for some constant $c\in\R\backslash\{0\}$.
\item If $\Lambda_0=0$, then 
\begin{equation}
\label{dec-Jacobi-field-log}
    0<c_1\le|y|^{\frac{N-2}{2}}|\zeta(y)|\le c_2\log|y|\qquad \text{for $|y|\ge R$,}
\end{equation}
for some $C_1,\,c_2,\,R>0$.
\end{enumerate}
This is true since 
$$\zeta(y)=r^2\partial_r(r^{-1} w)+O(r^{-\beta}w),\qquad\text{as $r\to\infty$},$$
where $y$ and $r$ are related through the relation $y=r\theta+w(r,\theta)\nu_C(y)$, $(r,\theta)\in(R,\infty)\times\Gamma$, $w$ is the function whose normal graph is $\Sigma$ and $\beta>0$ (see Theorem \ref{th-dec-w} Remark $2.2$ in \cite{HS}). We note that such a Jacobi field does not vanish a compact subset of $\Sigma$.
\end{remark}
We set 
\begin{equation}\notag
\bar{\nu}:=\inf\{\nu\in\R:\, \limsup_{|y|\to\infty}|y|^{-\nu}|\zeta(y)|=0\}.
\end{equation}
and $\bar{\delta}:=\frac{N-2}{2}+\bar{\nu}$. By Remark \ref{remark-behaviour-Jac-fields} if $\Sigma$ is asymptotic to a stable cone $C$ and does not intersect it, then $\bar{\nu}$ is well defined and $|\bar{\delta}|=\Lambda_0\ge 0$.\\

\begin{lemma}
If $\Sigma$ is a stable minimal hypersurface asymptotic to a cone $C$, then $C$ is stable.
\end{lemma}
\begin{proof}
If we assume by contradiction that $C$ is unstable, then there exists a function $\phi\in C^\infty_c(C\backslash sing(C))$ such that $\mathcal{Q}_C(\phi)<0$. Introducing a small scaling parameter $\epsilon>0$, by homogeneity it is possible to see that the new function $\phi_\epsilon(r,\theta):=\phi(\epsilon r,\theta)$ fulfils $$\mathcal{Q}_C(\phi)=\epsilon^{2-N}\mathcal{Q}_C(\phi)<0.$$
This shows that for any $\epsilon>0$, there exists a function $\phi_\epsilon\in C^\infty_c(C\backslash {\rm sing}(C))$ supported in $\Sigma_{\epsilon^{-1}}$ such that $\mathcal{Q}_C(\phi_\epsilon)<0$.\\

Now we fix $0<\rho<R$ and set $\varphi_\epsilon(y):=\phi_\epsilon(r,\theta)$, where $$y=r\theta+w(r,\theta)\nu_C(\theta)\in\Sigma_R.$$
Setting $\varphi_\epsilon(y)=|y|^{-\frac{N-2}{2}}u_\epsilon(y)$, we note that
\begin{equation}\notag
\begin{aligned}
\mathcal{Q}_\Sigma(\varphi_\epsilon)&=\int_\Sigma -\mathcal{L}u_\epsilon u_\epsilon |y|^{-N} d\sigma=\\
&=\int_\Gamma d\tilde{\sigma}\int_{\log R}^\infty -Lu_\epsilon u_\epsilon (1+o_\epsilon(1))d\tc d\tilde{\sigma}=\\
&=\int_C -J_C\phi_\epsilon \phi_\epsilon (1+o_r(1)) d\sigma_C<0,
\end{aligned}    
\end{equation}
where the operator $L$ is defined in Remark \ref{remark-product-metric} and $d\sigma_C$ is the volume element of the cone, since $L$ fulfils
$$\int_C -J_C\phi\, \phi\, d\sigma_C=\int_\Gamma d\tilde{\sigma}\int_{\log R}^\infty -Lu\,u\,d\tc d\tilde{\sigma} \qquad\forall \phi\in C^\infty_c(C\backslash {\rm sing}(C)),$$
where $\phi(r,\theta)=e^{-\frac{N-2}{2}\tc}u(\tc,\theta)$. This shows that $\Sigma$ is unstable, a contradiction.
\end{proof}
According to the properties of the Jacobi fields coming from dilations, we have an injectivity result for $A_\delta$ if $\Sigma$ is stable. 
\begin{proposition}
\label{prop-injective}
Let $\Sigma$ be a stable minimal hypersurface asymptotic to a cone $C$ which does not intersect $C$. Then $A_\delta$ is injective for any $\delta<\bar{\delta}$.
\end{proposition}


Proposition \ref{prop-injective} will be proved below. 
One of the purposes of this paper is to apply Proposition \ref{prop-injective} to the hypersurfaces constructed in Theorems \ref{th-Sigma} and \ref{th_Al}. In order to do so we introduce the concept of \textit{strictly area minimising cone}.
\begin{definition}
    We say that a cone $C\subset\R^{N+1}$ is strictly area-minimising (or strictly minimising) if there exists a constant C > 0 such that, for any $\epsilon>0$ small enough
    $$Area(C\cap B_1)\le Area (S)-C\epsilon^N,$$
for any hypersurface $S\subset\R^{N+1}\backslash B_\epsilon$ such that $\partial S=\Gamma$.
\end{definition}
It is proved in \cite{La} that the Lawson cone $C_{m,n}$ is strictly area minimising for $m,n\ge 3$ with $m+n=8$ or $m,n\ge 2$ with $m+n\ge 9$. As we mentioned in the introduction, the cones $C_{2,6}$ and $C_{6,2}$ are not minimising, so in particular they cannot be strictly minimising.
\begin{proposition}
\label{Prop-injectivity-part-case}
\begin{enumerate}
\item Let $\Sigma$ be one of the hypersurfaces constructed either in Theorem \ref{th-Sigma} with the assumption that $C$ is strictly area minimising or in Theroem \ref{th_Al}. Then $\bar{\delta}=\Lambda_0$ and hence $A_\delta$ is injective for $\delta<\Lambda_0$.
\item  Let $\Sigma$ be one of the hypersurfaces constructed in Theorem \ref{th-Sigma} with the assumption that $C$ is minimising but not strictly. Then $\bar{\delta}=-\Lambda_0$ and hence $A_\delta$ is injective for $\delta<-\Lambda_0$.
\end{enumerate}
\end{proposition}
\begin{proof}[Proof of Proposition \ref{Prop-injectivity-part-case}]
If $\Sigma$ is one of the hypersurfaces constructed in Theorem \ref{th-Sigma}, then in particular it is stable and it does not intersect $C$.\\ 

If $\Lambda_0>0$, due to Remark \ref{remark-behaviour-Jac-fields} and Theorem $3.2$ of \cite{HS}, the Jacobi field $\zeta(y):=y\cdotp\nu_\Sigma(y)$ fulfils (\ref{dec-Jacobi-field-slow}) if $C$ is strictly minimising, so that $\bar{\delta}=\Lambda_0$, or (\ref{dec-Jacobi-field-fast}) if $C$ is minimising but not strictly, so that $\bar{\delta}=-\Lambda_0$. Hence the result follows from Proposition \ref{prop-injective}.\\

If $\Lambda_0=0$, then Remark 
\ref{remark-behaviour-Jac-fields} yields that $\bar{\delta}=0$, therefore the result follows once again from Proposition \ref{prop-injective}.\\

If $\Sigma$ is one of the hypersurfaces constructed in Theroem \ref{th_Al}, then it is stable, because it has a positive Jacobi field, and it does not intersect the asymptotic cone $C_{m,n}$.\\

Moreover, by Remark \ref{remark-behaviour-Jac-fields} we have $\bar{\delta}=\Lambda_0=\sqrt{\left(\frac{N-2}{2}\right)^2-(N-1)}>0$ if $C_{m,n}$ is strictly minimising, which is the case for $m,n\ge 2$ with $N+1=m+n\ge 9$ or $m,n\ge 3$ with $N+1=m+n=8$, thus, for such values of $m$ and $n$, the conclusion follows from Proposition \ref{prop-injective}.\\

It remains to treat the cases $m=2$, $n=6$ and $m=6$, $n=2$. In such cases we will see in Section \ref{sec-Jacobi} that $\bar{\delta}=\Lambda_0=\sqrt{\left(\frac{N-2}{2}\right)^2-(N-1)}>0$ as well (see Proposition \ref{prop-fund-syst}) and Remark (\ref{rem-zeta}), which concludes the proof, due to Proposition \ref{Prop-injectivity-part-case}.

\end{proof}
Roughly speaking, if the asymptotic cone is strictly area-minimising, we have a better injectivity result, since we have injectivity on a space of functions with slower decay at infinity.\\

The remaining part of the Section will be devoted to the proof of Proposition \ref{prop-injective}. For this purpose, we set $$\Lambda_j:=\sqrt{\left(\frac{N-2}{2}\right)^2+\lambda_j},\qquad j\ge 0,$$
where $\lambda_0<\lambda_1\le \dots\le \lambda_j\to\infty$ are the eigenvalues of $-J_\Gamma$. Note that $\Lambda_j$ are nonnegative real numbers for any $j\ge 0$ provided (\ref{strict-stab-cone}) holds, actually positive for $j\ge 1$. The numbers $\pm\Lambda_j$ are known as the \textit{indicial roots} of $\Sigma$.

The proof of Proposition \ref{prop-injective} relies on Lemma $11.1.4$ of \cite{P}, which can be used to prove the following result.
\begin{lemma}
\label{lemma-Pacard-improvement}
Let $\delta,\delta'\in\R\backslash\cup_{j\ge 0} \{\pm\Lambda_j\}$ with $\delta'<\delta$. If $u\in \mathcal{W}^{2,2}_\delta(\Sigma,|y|^{-N})$ is such that $\mathcal{L}u=0$ and there are no indicial roots in $(\delta',\delta)$, then $u\in \mathcal{W}^{2,2}_{\delta'}(\Sigma,|y|^{-N})$.
\end{lemma}
\begin{proof}
The proof is very similar to the proof of Lemma $11.2.1$ of \cite{P}, however, for the sake of clarity, we give the outlines. Setting, for $R>0$ large enough,
\begin{equation}\notag
u(y)=\tilde{u}(\tc,\theta)\qquad\forall\,y=e^\tc\theta+w(e^\tc,\theta)\nu_C(\theta)\in \Sigma_R
\end{equation}
and taking a smooth cutoff function $\chi:\R\to\R$ such that $\chi=0$ in $(-\infty,\log(R+1))$ and $\chi=1$ in $(\log(R+2),\infty)$, we can see that $$L(\chi\tilde{u})=-\chi\mathcal{R}\tilde{u}+2\partial_\tc\chi\cdotp\partial_\tc\tilde{u}+\tilde{u}\partial_{\tc\tc}\chi\in e^{(\delta+\gamma-1)\tc}L^2(\Gamma\times\R).$$ Recalling that $\gamma<0$, applying Lemma $11.1.4$ of \cite{P} we have
$$\tilde{u}=\tilde{v}+\sum_{\bar{\delta}<\pm\Lambda_j<\delta}e^{\pm\Lambda_j\tc}\varphi_j(\theta)$$
with $\bar{\delta}=\min\{\delta',\delta+\gamma-1\}$ and $\tilde{v}e^{-\bar{\delta}\tc}\in W^{2,2}(\Gamma\times(R,\infty))$. Using that there are no indicial roots in $(\delta',\delta)$, we conclude the proof.
\end{proof}
Lemma \ref{lemma-Pacard-improvement} gives an improvement of the decay of the solutions to $\mathcal{L}u=0$, which correspond to the Jacobi fields of $\Sigma$, thanks to Lemma \ref{lemma-change-var}. It is similar to Lemma $11.2.1$ of \cite{P}, however it is not exactly the same since condition $8.1$ of \cite{P} is not fulfilled here.\\ 

Moreover, we need a result about the pointwise decay of the Jacobi fields.
\begin{lemma}
\label{lemma-point-est}
Let $\Sigma\subset\R^{N+1}$ be a minimal hypersurface asymptotic to a cone $C$ at infinity. Let $\delta\in\R$ and let $u\in \mathcal{W}^{2,2}_\delta(\Sigma,|y|^{-N})$ be such that $\mathcal{L}u=0$. Then $u\in C^{2,\alpha}_{loc}(\Sigma)$, for some $\alpha\in(0,1)$, and
\begin{equation}
\label{dec-u-delta'}
|u(y)|\le c(1+|y|)^{\delta}\qquad\forall\,y\in\Sigma
\end{equation}
for some constant $c>0$.
\end{lemma}
\begin{proof}[Proof of Lemma \ref{lemma-point-est}]
 In order to prove our claim, we write, for $R>0$ large enough,
\begin{equation}\notag
    u(y)=\tilde{u}(\tc,\theta)\qquad\forall\,y=e^\tc\theta+w(e^\tc,\theta)\nu_C(\theta)\in \Sigma_R.
\end{equation}
and note that, in the notations of Remark \ref{remark-product-metric}, taking a smooth cutoff function $\chi:\R\to[0,1]$ such that
$\chi\equiv 0$ in $(-\infty,\log (R+1))$ and $\chi\equiv 1$ in $(\log (R+2),\infty)$, we have
$$L(\tilde{u}\chi(\tc) e^{-\delta\tc})=-e^{-\delta\tc}\chi(\tc)\mathcal{R}\tilde{u}+2\partial_\tc\tilde{u}\partial_\tc(\chi(\tc)e^{-\delta\tc})+\tilde{u}\partial_{\tc\tc}(\chi(\tc)e^{-\delta\tc})$$
in $\R\times\Gamma$. Using the elliptic estimates (see \cite{GT}) and the fact that $u\in \mathcal{W}^{2,2}_\delta(\Sigma,|y|^{-N})$, which is implies that $\tilde{u}\chi(\tc) e^{-\delta\tc}\in W^{2,2}(\R\times\Gamma)$, a bootstrap argument shows that $\tilde{u}\chi(\tc)\in W^{2,q}_{loc}(\R\times\Gamma)$ for some $q>N$, which yields that $\tilde{u}\chi(\tc) e^{-\delta\tc}\in C^{1,\alpha}(\R\times\Gamma)$ for some $\alpha\in(0,1)$ and $\tilde{u}\chi(\tc) e^{-\delta\tc}$ is bounded. Moreover, the elliptic estimates applied to $\phi=|y|^{-\frac{N-2}{2}}u$ on the compact set $K_{R+2}$ show that $\phi$ is bounded in $K_{R+2}$, so in particular (\ref{dec-u-delta'}) holds. By the H\"{o}lder elliptic estimates (see \cite{GT}), it is possible to see that $\tilde{u}\chi(\tc) e^{-\delta\tc}\in C^{2,\alpha}_{loc}(\R\times\Gamma)$, which yields that $u\in C^{2,\alpha}_{loc}(\Sigma)$.\\

The estimate follows directly from the H\:{o}lder regularity since the constants in the estimates do not depend on the center of the ball.
\end{proof}

Furthermore, we need a refined version of the maximum principle on an unbounded domain of $\Sigma$, which parallels the maximum principle for possibly unbounded domains of $\R^N$ proved in \cite{BCN}.
\begin{lemma}[Maximum principle in possibly unbounded domains of $\Sigma_R$]
\label{max-pr-Sigma_tau}
Let $C\subset\R^{N+1}$ be a stable cone and let $\Sigma$ be a minimal hypersurface asymptotic to a cone $C$. Then there exists $R_0>0$ such that, if $R>R_0$,
$\Omega\subset\Sigma_R$ is open (not necessarily bounded) and $u\in C^2(\Omega)$ satisfies either
\begin{equation}
\label{dec-u-princ-max}
\begin{aligned}
    &|u(y)|\le c|y|^{\delta},\qquad\text{for some $\delta<\Lambda_0$ if $\Lambda_0>0$, or}\\
    &|u(y)|\le c(\log|y|)^\beta\qquad\text{for some $\beta\in(0,1)$ if $\Lambda_0=0$}
\end{aligned}    
\end{equation}
and
\begin{equation}\notag
\begin{aligned}
-\mathcal{L}u&\le 0\qquad\text{in $\Omega$}\\
u&\le 0\qquad\text{on $\partial\Omega$.}
\end{aligned}
\end{equation}
then $u\le 0$ in $\Omega$.  
\end{lemma}
We note that, in the statement of the maximum principle, we do not need $\Sigma\cap C=\emptyset$ and condition (\ref{dec-u-princ-max}) is relevant in case $\Omega$ is unbounded, otherwise it is trivially satisfied.
\begin{proof}
First we construct a smooth function $\phi>0$ such that $-\mathcal{L}\phi> 0$ in $\Sigma_R$ and $\phi(y)\to\infty$ as $|y|\to\infty$. For this purpose, if  $\Lambda_0>0$, we choose $\alpha\in(\delta,\Lambda_0)$, we consider a solution $\varphi_0>0$ to the eigenvalue problem $-J_\Gamma \varphi_0=\lambda_0 \varphi_0$ in $\Gamma$ 
and set $$\phi(y)=\tilde{\phi}(\tc,\theta):=e^{\alpha \tc}v(\theta)>0,\qquad y=e^\tc \theta+w(e^\tc,\theta)\nu_\Gamma(\theta) \in\Sigma_R.$$
By the decay of $w$ as $r\to\infty$ and Remark \ref{remark-product-metric}, we have
\begin{equation}\notag
\begin{aligned}
-\mathcal{L}\phi
&=-\Delta_\Gamma\tilde{\phi}-\partial_{\tc\tc} \tilde{\phi}-|A_\Gamma|^2\tilde{\phi}+\left(\frac{N-2}{2}\right)^2\tilde{\phi}+\mathcal{R}(e^{\alpha\tc} \varphi_0)=\\
&=e^{\alpha \tc}\left(-J_\Gamma \varphi_0-\alpha^2 \varphi_0+\left(\frac{N-2}{2}\right)^2\varphi_0+o(1)\varphi_0\right)=\\
&=e^{\alpha \tc}\varphi_0\left(\lambda_0+\left(\frac{N-2}{2}\right)^2-\alpha^2+o(1)\right)>0
\end{aligned}
\end{equation}
in $\Sigma_R$ provided $R>0$ is large enough.\\ 

If $\Lambda_0=0$, the same claim can be proved by setting $$\phi(y)=\tilde{\phi}(\tc,\theta):=\tc^\alpha \varphi_0(\theta)>0,\qquad\alpha\in(\beta,1).$$
In fact, with this choice we have $\phi(y)\to\infty$ as $|y|\to\infty$ and
\begin{equation}\notag
\begin{aligned}
-\mathcal{L}\phi &=-\Delta_\Gamma\tilde{\phi}-\partial_{\tc\tc} \tilde{\phi}-|A_\Gamma|^2\tilde{\phi}+\left(\frac{N-2}{2}\right)^2\tilde{\phi}+o(\tc^{\alpha-2}) \varphi_0=\\
&=-\alpha(\alpha-1)\tc^{\alpha-2}\varphi_0+\left(\lambda_0+\left(\frac{N-2}{2}\right)^2\right)\tc^\alpha\varphi_0+o(\tc^{\alpha-2})\varphi_0 \\
&=-\alpha(\alpha-1)\tc^{\alpha-2}\varphi_0+o(\tc^{\alpha-2})\varphi_0>0
\end{aligned}
\end{equation}
in $\Sigma_R$, for $R>0$ large enough.\\

Then, setting $\sigma(y):=\frac{u(y)}{\phi(y)}$, $\forall\,y\in\Omega$,  and using the behaviour of $u$ as $|y|\to\infty$ if $\Omega$ is unbounded, it is possible to see that
\begin{equation}\notag
\begin{aligned}
&|y|^N\diver(|y|^{2-N}\nabla_\Sigma\sigma)+\frac{\mathcal{L}\phi}{\phi}\sigma+2|y|^2\frac{\nabla_{\Sigma}\sigma\cdotp\nabla_{\Sigma}\phi}{\phi}\\
&=\frac{|y|^N\diver(|y|^{2-N}\nabla_\Sigma u)}{\phi}-\frac{2}{\phi^2}|y|^2\nabla_{\Sigma}u\cdotp\nabla_{\Sigma}\phi-\frac{u}{\phi^2}|y|^N\diver(|y|^{2-N}\nabla_\Sigma\phi)\\
&+2\frac{u}{\phi}|y|^2\frac{|\nabla_{\Sigma}\phi|^2}{\phi^2}+\frac{\mathcal{L}\phi}{\phi}\frac{u}{\phi}+2|y|^2\frac{\nabla_{\Sigma}(u/\phi)\cdotp\nabla_\Sigma\phi}{\phi}\\
&=\frac{\mathcal{L}u}{\phi}\ge 0\qquad\text{in $\Sigma_R$,}\\
&\sigma\le 0\,\text{on $\partial\Sigma_R$.}
\end{aligned}
\end{equation}
and $\lim_{|y|\to\infty}\sigma(y)=0$ if $\Omega$ is unbounded.\\ 

Next we prove that $\sigma\le 0$ in $\Sigma_R$. If we assume by contradiction that $\sigma$ has a positive maximum at some point $y_0\in\Sigma_R$, we have
$$0\le \frac{\mathcal{L}u(y_0)}{\phi(y_0)}<\left(|y|^N\diver(|y|^{2-N}\nabla_\Sigma\sigma)\right)\bigg|_{y=y_0}=|y_0|^2\Delta_\Sigma \sigma(y_0)\le 0,$$
which is a contradiction, hence $\sigma\le 0$ in $\Sigma_R$.
\end{proof}
\begin{remark}
    In the proof of Lemma \ref{max-pr-Sigma_tau} we see the importance of the change of variables $\phi=|y|^{-\frac{N-2}{2}}u$. In fact the operator $\mathcal{L}$ allows us to define the supersolution $\phi$ as a product of a function on $(\tc_0,\infty)$ and a function on $\Gamma$.
\end{remark}

Now we can prove Proposition \ref{prop-injective}.
\begin{proof}[Proof of Proposition \ref{prop-injective}]
\begin{enumerate}

\item First we treat the case $\Lambda_0>0$.
\begin{itemize}
\item In case $\zeta$ fulfils (\ref{dec-Jacobi-field-slow}) and $u$ is a Jacobi field in $\mathcal{W}^{2,2}_\delta(\Sigma,|y|^{-N})$ with $\delta\in(-\Lambda_0,\Lambda_0)$, then by Lemma \ref{lemma-Pacard-improvement}, $u\in \mathcal{W}^{2,2}_{\delta'}(\Sigma,|y|^{-N})$ for any $\delta'\in(-\Lambda_0,\delta)$, so in particular $u\in \mathcal{W}^{2,2}_0(\Sigma,|y|^{-N})$. Moreover, by Lemma \ref{lemma-point-est}, $u\in C^{2,\alpha}_{loc}(\Sigma)$ and fulfils the decay estimate (\ref{dec-u-delta'}) for any $\delta'\in(-\Lambda_0,\delta)$. If $\delta\le-\Lambda_0$, then in particular $u\in \mathcal{W}^{2,2}_0(\Sigma,|y|^{-N})$ and it fulfils (\ref{dec-u-delta'}), by Lemma \ref{lemma-point-est}.\\

Since $u\in \mathcal{W}^{2,2}_0(\Sigma,|y|^{-N})\subset \mathcal{W}^{1,2}_0(\Sigma,|y|^{-N})$ and $\Sigma$ is stable, $|u|$ is a Jacobi field too, in fact 
\begin{equation}\notag
\begin{aligned}
0&=\int_\Sigma -\mathcal{L}u\, u|y|^{-N} d\sigma=\int_\Sigma -\mathcal{L}|u|\, |u||y|^{-N} d\sigma=\\
&=\inf\left\{\int_\Sigma -\mathcal{L}v\, v|y|^{-N} d\sigma:\,v\in \mathcal{W}^{1,2}_0(\Sigma,|y|^{-N})\right\},
\end{aligned}
\end{equation}

Here the stability of $\Sigma$ plays a crucial role, since it guarantees that the quadratic form is non-negative.\\

Hence, by the strong maximum principle, the fact that $|u|\ge 0$ implies that either $u\equiv 0$, which concludes the proof, or $|u|$ never vanishes in $\Sigma\backslash\{0\}$, 
thus we can assume without loss of generality that the second possibility holds. Setting $v_0(y):=|y|^{\frac{N-2}{2}}|\zeta(y)|$, we have $v_0>0$ in $\Sigma_{R_0}$ provided $R_0$ is large enough, due to Remark \ref{remark-behaviour-Jac-fields}, since $\Sigma$ does not intersect $C$. Then, for $R>R_0$ large enough, we can choose $\epsilon=\epsilon(R)>0$ (small) such that
\begin{equation}\notag
\begin{aligned}
    -\mathcal{L}(u-\epsilon v_0)&=0\qquad\text{in $\Sigma_R$}\\
    u-\epsilon v_0&\ge 0\qquad\text{on $\partial\Sigma_R$.}
\end{aligned}
\end{equation}
and apply the maximum principle stated in Lemma \ref{max-pr-Sigma_tau} to conclude that $u\ge \epsilon v_0$ in $\Sigma_R$. However, this contradicts (\ref{dec-u-delta'}), since $v_0(y)=|y|^{\Lambda_0}(1+o(1))$ as $|y|\to\infty$ 
\item In case $\zeta$ fulfils (\ref{dec-Jacobi-field-fast}), we observe that any Jacobi field $u\in \mathcal{W}^{2,2}_\delta(\Sigma,|y|^{-N})$ with $\delta<-\Lambda_0$ fulfils $u\in \mathcal{W}^{2,2}_0(\Sigma,|y|^{-N})\cap C^{2,\alpha}_{loc}(\Sigma)$ and the pointwise decay estimate (\ref{dec-u-delta'}). Then the proof follows the same outlines as above, that is we compare $u$ with the barrier $v_0(y):=|y|^{-\frac{N-2}{2}}|\zeta(y)|$ and, by the maximum principle stated in Lemma \ref{max-pr-Sigma_tau} we conclude that $u\ge cv_0=c|y|^{-\Lambda_0}(1+o(1))$ in $\Sigma_R$ for $R>0$ large enough, which contradicts (\ref{dec-u-delta'}).
\end{itemize}
\item If $\Lambda_0=0$ and  $u$ is a Jacobi field in $\mathcal{W}^{2,2}_\delta(\Sigma,|y|^{-N})$ with $\delta<0$, then $u\in \mathcal{W}^{2,2}_0(\Sigma,|y|^{-N})$ and it fulfils (\ref{dec-u-delta'}). On the other hand, taking $R$ large enough and comparing $u$ with the barrier $v_0(y):=|y|^{\frac{N-2}{2}}|\zeta(y)|$ as above, the maximum principle in $\Sigma_R$ (see Lemma \ref{max-pr-Sigma_tau}) yields that $$u(y)\ge c>0\qquad\text{in $\Sigma_R$},$$
for some constant $c>0$, a contradiction.
\end{enumerate}
\end{proof}
\section{Stability properties}\label{sec-strict-stability}
In this section we will deal with $O(m)\times O(n)$-invariant hypersurfaces. More precisely, we take $m,n\in\mathbb{N}$ such that $m,n \geq 2$ and such that $N+1=m+n \geq 8$ and we assume that $\Sigma \subset \R^{N+1}$ is one of the hypersurfaces constructed in Theorem \ref{th_Al}.\\ 

We consider the space
\begin{equation}
\label{def-X}
X:=\left\{\phi\in H^1_{loc}(\Sigma):\,|\nabla_\Sigma\phi|\in L^2(\Sigma),\,|y|^{-1}\phi\in L^2(\Sigma)\right\}.
\end{equation}
Note that the function $y\in\Sigma\mapsto|y|^{-1}\in(0,\infty)$ is smooth and bounded, since $dist(\Sigma,\{0\})=1$, and $|A_\Sigma|\phi\in L^2(\Sigma)$, for any $\phi\in X$, because $|A_\Sigma|^2\le c|y|^{-2}$.\\

In this section, we prove the following statement.

\begin{proposition}
\label{prop-strict-stability}
Let $\Sigma$ be one of the minimal hypersurfaces constructed in Theorem \ref{th_Al}. Then there exists $\lambda>0$ such that
$$\int_\Sigma\big(|\nabla_\Sigma\phi|^2-|A_\Sigma|^2\phi^2\big)d\sigma\ge \lambda \int_\Sigma|A_\Sigma|^2\phi^2 d\sigma\qquad\forall\,\phi\in X.$$
\end{proposition}
In particular, Proposition \ref{prop-strict-stability} proves Theorem \ref{th-strict-stability}.\\

Remark $3.3$ point $(3)$ of \cite{HS} asserts that if $\Sigma$ is strictly stable and the asymptotic cone $C$ is minimizing, then $C$ is strictly stable too. Here we prove that the opposite implication is also true, at least in the case of the Lawson cones, without any minimality assumption about the cone.
\begin{proof}
Considering the function $u:\Sigma \to \R$ defined by $\phi(y)=|y|^{-\frac{N-2}{2}}u(y)$ and setting $H:=\mathcal{W}^{1,2}_0(\Sigma,|y|^{-N})$, the statement is equivalent to prove that
$$\int_\Sigma (|\nabla_{\Sigma}u|^2|y|^2-V u^2) |y|^{-N}d\sigma\ge \mu\int_\Sigma u^2 |y|^{-N}d\sigma\qquad\forall\, u\in H,$$
for some $\mu>0$, where $V:=|y|^2|A_\Sigma|^2+|y|^{\frac{N+2}{2}}\Delta_\Sigma(|y|^{-\frac{N-2}{2}})$. 



In other words, we need to prove that
$$\mu:=\inf\left\{\int_\Sigma (|\nabla_{\Sigma}u|^2|y|^2-V u^2) |y|^{-N}d\sigma:\,u\in H,\,\int_\Sigma u^2 |y|^{-N}d\sigma=1\right\}>0.$$
We already know that $\mu\ge 0$, due to the stability of $\Sigma$. It remains to exclude the case $\mu=0$. If we assume by contradiction that $\mu=0$, we can find a sequence $u_k\in H$ such that $\|u_k |y|^{-\frac{N}{2}}\|_{L^2(\Sigma)}=1$ and $$\int_\Sigma (|\nabla_{\Sigma}u_k|^2|y|^2-V u_k^2) |y|^{-N}d\sigma\to 0.$$
As a consequence $u_k$ is bounded in $H$, so that it admits a subsequence converging weakly in $H$ and strongly in $L^2_{loc}(\Sigma)$ to some function $u\in H$, due to Rellich-Kondrakov's Theorem. If we assume that there exist $R,\,\delta,\,k_0>0$ such that, for any $k\ge k_0$, we have
\begin{equation}
\label{u_k-concentrated}
\int_{K_R} u_k^2 |y|^{-N}d\sigma> \delta>0,
\end{equation}
we get that $u\ne 0$. Moreover, since $\Sigma$ is $O(m)\times O(n)$-invariant, $|A_\Gamma|^2=N-1$ is explicit, so that there exists $\bar{R}\ge R$ such that 
\begin{equation}
\label{fix-tau}
V<-\frac{1}{2}\left(\left(\frac{N-2}{2}\right)^2-(N-1)\right)<0\qquad\text{on $\Sigma_{\bar{R}}$,}
\end{equation}
so that
$$\int_{\Sigma_{\bar{R}}}-Vu^2|y|^{-N} d\sigma\le \liminf_{k\to\infty}\int_{\Sigma_{\bar{R}}}-Vu_k^2 |y|^{-N}d\sigma$$
by the Fatou Lemma. As a consequence, using the strong convergence in $K_{\bar{R}}$ and the weak lower semicontinuity of the seminorm
$$v\in H\mapsto \int_\Sigma |\nabla_\Sigma v|^2|y|^{2-N}d\sigma,$$
and the strong convergence in $L^2_{loc}(\Sigma)$, we get
$$0\le \int_\Sigma (|\nabla_{\Sigma}u|^2|y|^2-V u^2)|y|^{-N} d\sigma\le \liminf_{k\to\infty}\int_\Sigma (|\nabla_{\Sigma}u_k|^2|y|^2-V u_k^2)|y|^{-N} d\sigma=\mu=0,$$


\color{black}which yields that $$0=\int_\Sigma-\mathcal{L}u\, u |y|^{-N}d\sigma=\inf\left\{\int_\Sigma-\mathcal{L}v\, v|y|^{-N}
d\sigma:\,v\in H\right\},$$
since $\Sigma$ is stable, which yields that $\mathcal{L}u=0$. Moreover, since $u\in L^2_0(\Sigma,|y|^{-N})$, Proposition \ref{prop-injective} ($\Lambda_0>0$ and $\delta=0$) yields that $u=0$, a contradiction since $$\int_{K_R}u^2|y|^{-N} d\sigma\ge\delta>0.$$

Therefore (\ref{u_k-concentrated}) is false, so that there exists a sequence $k_j\to\infty$ fulfilling
$$\int_{\Sigma_j}u_{k_j}^2 |y|^{-N} d\sigma\ge 1-\frac{1}{j}.$$
As a consequence, choosing $j>R$, with $R$ as in (\ref{fix-tau}), we have 
\begin{equation}\notag
\begin{aligned}
&\int_\Sigma (|\nabla_{\Sigma}u_k|^2|y|^2-V u_k^2)|y|^{-N} d\sigma\ge\int_\Sigma -V u_k^2|y|^{-N} d\sigma\\
&=\int_{\Sigma_j} -V u_k^2 |y|^{-N} d\sigma+\int_{K_j} -V u_k^2 |y|^{-N} d\sigma\ge c+o(1),
\end{aligned}
\end{equation}
for some constant $c>0$, a contradiction.
\end{proof}
\begin{remark}
    In the proof of Proposition \ref{prop-strict-stability} we see once again the importance of the change of variables $\phi=|y|^{-\frac{N-2}{2}}u$, since the potential $V$ is uniformly positive outside a compact subset of $\Sigma$.
\end{remark}
Now we will consider the hypersurfaces constructed in Theorem \ref{th-Sigma-low-dim} and we will compute their Morse index. More precisely, we will prove the following result.
\begin{proposition} 
\label{prop-quadr-Morse}
Let $\Sigma$ be one of the hypersurfaces constructed in Theorem \ref{th-Sigma-low-dim}. Then there exists $\lambda>0$ and a linear infinite dimensional subspace $X\subset C^\infty_c(\Sigma)$ such that
\begin{equation}
\label{Q-def-neg}
\mathcal{Q}_\Sigma(\phi)\le -\lambda\int_\Sigma |A_\Sigma|^2\phi^2 \qquad\forall\,\phi\in C^\infty_c(\Sigma).
\end{equation}
\end{proposition}

Proposition \ref{prop-quadr-Morse} implies that the Morse index of such hypersurfaces is infinite, so that Theorem \ref{th-Morse-infinite} is true. 
\begin{proof}
Since $\Sigma=\Sigma_{m,n}$ is asymptotic to the Lawson cone $C_{m,n}$ at infinity, then it is enough to find a subspace $Y$ of compactly supported functions on the cone vanishing on a sufficiently large ball such that
\begin{equation}
\label{Q-C-def-neg}
\mathcal{Q}_{C_{m,n}}(\varphi):=\int_{C_{m,n}} (|\nabla_{C_{m,n}}\varphi|^2-|A_{C_{m,n}}|^2\varphi^2)<-\lambda\int_{C_{m,n}}|A_{C_{m,n}}|^2\varphi^2\qquad\forall\,\varphi\in Y.
\end{equation}
In order to do so, we consider the eigenvalues problem on the cone 
\begin{equation}
\label{ev-problem-cone}
\varphi_{rr}+\frac{N-1}{r}\varphi_r+\frac{N-1}{r^2}\varphi=\lambda\frac{N-1}{r^2}\varphi
\end{equation}
for $\lambda>0$ given. The purpose is to show that, for $3\le N\le 6$ and $\lambda>0$ small enough (depending on $N$), there exists a solution which, outside a compact set, changes sign infinitely may times.\\ 

First we look for polynomial solutions of the form $\varphi(r):=r^\beta$. A computation shows that $\varphi$ is a solution if and only if
$$\beta=\beta_\pm:=-\frac{N-2}{2}\pm\sqrt{\left(\frac{N-2}{2}\right)^2-(N-1)(1-\lambda)}.$$
Since we are interested in oscillating solutions we want
$$-\Lambda^2(\lambda,N):=\left(\frac{N-2}{2}\right)^2-(N-1)(1-\lambda)<0$$
Since we want to solve the problem with $\lambda>0$ and $3\le N\le 6$, this is possible if $$0<\lambda<\lambda_0(N)=\frac{1}{N-1}\left(N-1-\left(\frac{N-2}{2}\right)^2\right).$$ 
The corresponding solutions read
$$\varphi_\pm(r)=r^{-\frac{N-2}{2}}(\cos(\Lambda(\lambda,N)\log r)\pm i\sin(\Lambda(\lambda,N)\log r)).$$
Taking, for example, the real part of $\varphi_+$, we have the required solution $$\varphi(r):=r^{-\frac{N-2}{2}}\cos(\Lambda(\lambda,N)\log r),$$
which vanishes on a sequence $r_k\to\infty$. Setting $\varphi_k:=\chi_{[r_k,r_{k+1}]}(r)\varphi(r)$ and $$Y:=span\{\varphi_k:\, k\ge k_0\},$$
we have the required space of functions fulfilling (\ref{ev-problem-cone}). Since the functions $\varphi_k$ have disjoint supports, then $Y$ has infinite dimension. Approximating $\varphi_k$ in the $H^1(0,\infty)$-norm by smooth functions $\tilde{\varphi}_k:(0,\infty)\to\R$ supported in $(r_k,r_{k+1})$ and setting $\phi_k(y)=\tilde{\varphi}_k(s)$, for $y=(a(s)\x,b(s)\y)\in\Sigma$, we have the statement.
\end{proof}
Proposition \ref{prop-quadr-Morse} actually proves more than Theorem \ref{th-Morse-infinite}.\\ 

Given a subgroup $S\subset O(N+1)$ and an $S$-invariant minimal hypersurface $\Sigma$, we define the $S$-invariant Morse index of $\Sigma$ as
$$Morse(\Sigma,S):=\sup\{\dim(X):\,X\text{ subspace of }C^\infty_c(\Sigma,S):\,\mathcal{Q}_\Sigma(\phi)<0,\,\forall\,\phi\in X\backslash\{0\}\},$$
where $C^\infty_c(\Sigma,S)$ is the subspace of $S$-invariant compactly supported smooth functions on $\Sigma$. We note that
\begin{equation}
    Morse(\Sigma,S)\le Morse(\Sigma).
\end{equation}
\begin{remark}
    The proof of Proposition \ref{prop-quadr-Morse} actually shows that the hypersurface $\Sigma_{m,n}$ constructed in Theroem \ref{th-Sigma-low-dim} actually has infinite $O(m)\times O(n)$-invariant Morse index, for any $m,n\ge 2$ with $m+n\le 7$.
\end{remark}

\section{Nondegeneracy}\label{sec-Jacobi}

Here we consider one of the hypersurfaces constructed in Theorem \ref{th_Al}, which is $O(m)\times O(n)$-invariant and asymptotic to $C_{m,n}$ at infinity, with $m,\,n\ge 2$, $m+n=N+1\ge 8$, and we give a complete characterisation of the bounded Kernel of its Jacobi operator $J_\Sigma=\Delta_\Sigma+|A_\Sigma|^2$. We will see that the strict stability of $\Sigma$ plays a crucial role.
\begin{theorem}
\label{th-Jac-Sigma}
Let $\Sigma$ be one of the hypersurfaces constructed in Theorem \ref{th_Al}. Then the space $K(\Sigma)$ of bounded Jacobi fields of $\Sigma$ has dimension $N+2$. Moreover, the Jacobi fields in $K(\Sigma)$ are smooth and $K(\Sigma)$ has an $(N+1)$-dimensional subspace $T$ of non $O(m)\times O(n)$-invariant functions. 
\end{theorem}
\begin{remark}\label{rmk:space_transl}
\begin{itemize}
\item The Jacobi fields in $T$ correspond to translations and $T$ has dimension $N+1$. The Jacobi field $y\cdotp \nu_\Sigma(y)\in K(\Sigma)\backslash T$ corresponds to dilation and it is $O(m)\times O(n)$-invariant.

\item In particular, all these Jacobi fields are geometric, hence Theorem \ref{th-Jac-Sigma} implies Theorem \ref{th-nondegeneracy}, case $(1)$.
\end{itemize}
\end{remark}

Since the hypersurface $\Sigma^+_{m,n}$ can be obtained by $\Sigma^-_{n,m}$ through an orthogonal transformation $\sigma\in O(N+1)$, namely 
\begin{equation}\notag
    \Sigma^+_{m,n}=\sigma(\Sigma^-_{n,m}),\qquad\sigma:(x,y)\in\R^m\times\R^n\mapsto(y,x)\in\R^n\times\R^m,
\end{equation}
we can restrict ourselves to consider $\Sigma:=\Sigma^-_{m,n}$ and consider all the possible cases for $m,n$. We use the notation $y:=(a(s)\x,b(s)\y)\in\Sigma=\Sigma^-_{m,n}$ introduced in Section \ref{sec-notations}.\\
We consider an orthonormal basis $\{u_i\}_{i\in\N}$ of $L^2(S^{m-1})$ consisting of eigenfunctions of $-\Delta_{S^{m-1}}$ and an orthonormal basis $\{v_j\}_{j\in\N}$ of $L^2(S^{n-1})$ consisting of eigenfunctions of $-\Delta_{S^{n-1}}$. We denote the corresponding eigenvalues by $\mu^m_i$ and $\mu_i^n$ respectively, being $$0=\mu_0^l<\mu_1^l\le\dots\le\mu_i^l\to\infty\qquad i\to\infty, \qquad l=m,n.$$

$\mu_0^l=0$ is simple. $\mu_1^l=\cdots=\mu_l^{l}=l-1$. These eigenvalues are considered with repetitions.\\

Using a formal Fourier expansion $$\phi(s,\x,\y):=\sum_{i,j=1}^\infty \phi_{ij}(s)u_i(\x)v_j(\y)\qquad\phi_{ij}(s):=\int_{S^{m-1}\times S^{n-1}}\phi(s,\x,\y)u_i(\x)v_j(\y)d\x d\y$$
it is possible to see that $\bar{\phi}(y):=\phi(s,\x,\y)$ is a Jacobi field of $\Sigma$ if and only if
\begin{equation}
\label{eq-Jacobi-modos}
\partial_{ss} \phi_{ij}+\alpha(s)\partial_s \phi_{ij}+\left(\beta(s)-\frac{\mu_i^m}{a^2(s)}-\frac{\mu_j^n}{b^2(s)}\right)\phi_{ij}=0 \qquad\forall\, i,j\ge 0
\end{equation}
Introducing an Emden-Fowler change of variables $s=e^t$ and writing $\phi_{ij}(s)=w_{ij}(t)p(t)$, where $p$ is defined in \cite{AKR} and recalled above in (\ref{def-p(t)}), the equation for $w_{ij}$ is 
\begin{equation}
\label{eq-Jacobi-modos-EF}
\partial_{tt}w_{ij}+V_{ij}(t)w_{ij}=0,\qquad V_{ij}(t)=V(t)-\left(\frac{\mu_i^m}{a^2(e^t)}+\frac{\mu_j^n}{b^2(e^t)}\right)e^{2t},\,i,j\ge 0
\end{equation}
where $V$ is the potential defined in \cite{AKR} and recalled above in (\ref{def-V(t)}).\\ 

In order to study equations (\ref{eq-Jacobi-modos-EF}), we need the following general Lemma.
\begin{lemma}
\label{lemma-syst}
Let $\lambda\ge 0$ and $V\in C^\infty(\R)$ be a potential such that 
\begin{equation}
\label{V-infty}
|V(t)+\lambda^2|\le c e^{\eta t}\qquad\text{ for $t<t_0$, $\eta>0,\,c>0$}
\end{equation}
Then the equation
\begin{equation}
\label{eq-model-EF}
\partial_{tt}w+V(t)w=0
\end{equation}
admits two linearly independent entire solutions $w^\pm$ such that
\begin{enumerate}
\item \label{est-lambda>0} 
$w^\pm\sim e^{\pm\lambda t},\,\partial_t w\sim \pm\lambda e^{\pm\lambda t}$ as $t\to-\infty$ if $\lambda>0$. 
More precisely, the remainder fulfils
\begin{equation}
|w^\pm(t)e^{\mp\lambda t}-1|+|\partial_t w^\pm(t)\lambda^{-1}e^{\mp\lambda t}-1|\le
\begin{cases}
ce^{\min\{2\lambda,\eta\}t}\text{ if $2\lambda\ne \eta$}\\
ce^{2\lambda t}|t|\text{ if $2\lambda=\eta$}
\end{cases}
\qquad\forall,\ t<t_0.
\end{equation}
\item \label{est-lambda=0} 
\begin{equation}\notag
\begin{aligned}
w^-&=-t+O(1),\,\partial_t w^-=-1+O(te^{\eta t}),\\  
w^+&=1+O(t^{-1}),\, \partial_t w^+=O(te^{\eta t})
\end{aligned}
\end{equation}
as $t\to-\infty$ 
if $\lambda=0$.
\end{enumerate}
\end{lemma}
\begin{proof}
\begin{enumerate}
\item We consider the case $\lambda>0$. In order to find a solution $w^-$ with the required asymptotic behaviour at minus infinity, we look for a solution of the form $w^-(t):=x(t)e^{-\lambda t}$ and we rewrite equation (\ref{eq-Jacobi-modos-EF}) as a system
\begin{equation}
\label{syst-xy}
\begin{cases}
x_t&=e^{2\lambda t}y\\
y_t&=-q(t)e^{-2\lambda t}x
\end{cases}
\end{equation}
where we have set $V(t):=-\lambda^2+q(t)$. In order to prescribe the asymptotic behaviour of $w^-$ at $-\infty$, we choose $t_0<0$ small enough and solve the Cauchy problem associated to system (\ref{syst-xy}) with initial values $x(t_0)=1$ and $y(t_0)=0$ and from Fubini's Theorem we get
\begin{equation}\notag
x(t)=1+\int_t^{t_0} x(\zeta)q(\zeta)\frac{1-e^{-2\lambda(\zeta-t)}}{2\lambda}d\zeta.
\end{equation}
We estimate for $t\leq t_0$,
$$
\begin{aligned}
|x(t)| & \leq 1 + \int_{t}^{t_0} |x(\zeta)||q(\zeta)| \frac{1-e^{-2\lambda(\zeta-t)}}{2\lambda}d\zeta,\\
&\leq 1 + \int_{t}^{t_{\sigma}} |x(\zeta)|\frac{|q(\zeta)|}{2\lambda}d\zeta.
\end{aligned}
$$
Using the Gronwall inequality and the behaviour of $q$ at $-\infty$,
\begin{equation*}
|x(t)|\le\exp\bigg(\int_t^{t_0}\frac{|q(\zeta)|}{2\lambda}d\zeta\bigg)\le \exp\bigg(\int_{-\infty}^{t_0}\frac{|q(\zeta)|}{2\lambda}d\zeta\bigg)<\infty
\end{equation*}
for $t\leq t_0$. We conclude that $x\in L^{\infty}(-\infty, t_0)$. Moreover, taking $\tilde{t}<t<t_0$, we can see that
\begin{equation}
\label{x-lim}
\begin{aligned}
&|x(t)-x(\tilde{t})|=\left|\int_{\tilde{t}}^t e^{2\lambda s} ds\int_s^{t_0}q(\tau) e^{-2\lambda\tau}d\tau\right|\\
&\le 
\begin{cases}
c\int_{\tilde{t}}^t e^{2\lambda s}ds,\,\eta>2\lambda\\
c\int_{\tilde{t}}^t e^{\eta s}ds,\,\eta<2\lambda\\
\int_{\tilde{t}}^t e^{2\lambda s}(t_0-s)ds,\,\eta=2\lambda
\end{cases}
\to 0\qquad\text{as $t,\,\tilde{t}\to-\infty$,}
\end{aligned}
\end{equation}
so that there exists $\underline{x}:=\lim_{t\to-\infty} x(t)\in\R$. Moreover $\underline{x}>0$ provided $t_0$ is small enough. Using that $x$ is bounded on $(-\infty,t_0)$, it is possible to estimate $y$ in the form
\begin{equation}\notag
|y(t)|\le\int_t^{t_0}|q(\tau)|e^{-2\lambda\tau}d\tau\le\int_t^{t_0}e^{(\eta-2\lambda)\tau}d\tau\le
\begin{cases}
|t-t_0|,\,\eta=2\lambda\\
e^{(\eta-2\lambda) t},\,\eta<2\lambda\\
c,\,\eta>2\lambda
\end{cases}
\end{equation}
so that $\partial_t x(t)\to 0$ as $t\to-\infty$. Hence the asymptotic behaviour of $w^-$ at $-\infty$ is given by $$w^-\sim \underline{x}e^{-\lambda t},\,\partial_t w^-\sim-\lambda\underline{x}e^{-\lambda t}\qquad\text{as $t\to-\infty$.}$$
Multiplying by a constant, we can assume that $\underline{x}=-1$. 
Now we set $$w^+(t):=w^-(t)\int_{-\infty}^t \frac{ds}{(w^-)^2(s)}\qquad\forall\, t<t_0$$
and we note that $w^+$ is a solution to (\ref{eq-model-EF}) on $(-\infty, t_0)$, $w^+$ and $w^-$ are linearly independent and, multiplying if necessary by a constant, the required asymptotic behaviour at $-\infty$. 


\item Now we consider the case $\lambda=0$. We study equation (\ref{eq-model-EF}) in $(-\infty,t_0)$, with $t_0<0$ and $w^-(t)=tx(t)$, so that $x$ has to satisfy the equation
$$\partial_t(t^2\partial_t x)=-q(t)t^2 x(t),$$
which is equivalent to the system
\begin{equation}
\label{syst-neg-t}
\begin{cases}
\partial_t x=t^{-2}y,\qquad x(t_\sigma)=1\\
\partial_t y=-q(t)t^2 x,\qquad y(t_\sigma)=0
\end{cases}
\end{equation}
in $(-\infty,t_0)$. Integrating (\ref{syst-neg-t}) and using the initial conditions
\begin{equation}
\label{syst-xy-int}
x(t)=1-\int_t^{t_0}\tau^{-2} y(\tau)d\tau,\qquad y(t)=\int_t^{t_0} q(\tau)\tau^2 x(\tau)d\tau,
\end{equation}
thus, by Fubini's Theorem,
\begin{equation}\notag
\begin{aligned}
x(t)&=1-\int_t^{t_0} \tau^{-2}\int_\tau^{t_0} q(\xi)\xi^2 x(\xi)d\xi d\tau\\
&=1-\int_t^{t_0}q(\xi)\xi^2 x(\xi)\int_t^\xi \tau^{-2}d\tau d\xi,
\end{aligned}
\end{equation}
so that
$$|x(t)|\le 1-\int_t^{t_0}|q(\xi)\xi x(\xi)|d\xi$$
therefore, by The Gronwall inequality,
\begin{equation}\notag
|x(t)|\le\exp\left(-\int_t^{t_0}|q(\xi)\xi| d\xi\right)<\infty.
\end{equation}
hence $x\in L^\infty(-\infty,t_0)$ and, arguing as in point (\ref{est-lambda=0}), it is possible to see that there exists $\underline{x}:=\lim_{x\to-\infty} x(t)>0$. Using (\ref{syst-xy}), we can see that $y=\underline{x}(c+O(t^2e^{\eta t}))\in L^\infty(-\infty,t_0)$, therefore using (\ref{syst-xy}) we have
$$\partial_t x=\underline{x}ct^{-2}+O(e^{\eta t}).$$
The fundamental Theorem of calculus yields that $$x=\underline{x}-c\underline{x}t^{-1}+O(e^{\eta t}).$$
In terms of $w^-$, this yields that
\begin{equation}\label{eqn:asympt_w-}
w^-(t)= \underline{x}t-\underline{x}c+O(te^{\eta t}),\,\partial_t w^-(t)=\underline{x}+O(te^{\eta t}) \qquad t\to-\infty.
\end{equation}
Multiplying by a constant, we can assume that $\underline{x}=-1$. 
Another linearly independent solution to (\ref{eq-model-EF}) is given by
\begin{equation}\notag
w^+(t)=w^-(t)\int_{-\infty}^t (w^-(\tau))^{-2} d\tau.
\end{equation}
A careful computation using \eqref{eqn:asympt_w-} shows using that $w^+$ has the required asymptotic behaviour as $t\to-\infty$.
\end{enumerate}
\end{proof}
We are interested in the fundamental system of (\ref{eq-Jacobi-modos}), described in the following Proposition.

\begin{proposition}
\label{prop-fund-syst}
Let $\Sigma=\Sigma^-_{m,n}$ one of the hypersurfaces constructed in Theorem \ref{th_Al} and let $(i,j)\in\N\times\N$. Then equation (\ref{eq-Jacobi-modos}) has a fundamental system $\{\phi^\pm_{ij}\}$ satisfying
\begin{equation}\notag
\partial_s^{(l)}\phi^\pm_{ij}=
\begin{cases}
O\left(s^{-l-\frac{N-2}{2}\pm\sqrt{\left(\frac{N-2}{2}\right)^2-(N-1)+\nu_{ij}}}\right)\qquad s\to\infty\\
O\left(s^{-l-\frac{n-2}{2}\pm\sqrt{\left(\frac{n-2}{2}\right)^2+\mu_j^n}}\right)\qquad s\to 0^+
\end{cases}
\end{equation}
if $i\ge 0$ and $j\ge 1$, where $\nu_{ij}:=(N-1)\left(\frac{\mu_i^m}{m-1}+\frac{\mu_j^n}{n-1}\right)$, and 
\begin{equation}\notag
\begin{aligned}
\partial_s^{(l)}\phi^\pm_{i0}&=O\left(s^{-l-\frac{N-2}{2}\pm\sqrt{\left(\frac{N-2}{2}\right)^2-(N-1)+\nu_{i0}}}\right)\qquad s\to\infty,\,l\in\{0,1\}\\
\phi^+_{i0}(s)&=O(1),\, \partial_s\phi^+_{i0}=o(1)\qquad s\to 0^+\\
\phi^-_{i0}(s)&=O(s^{2-n})\,\text{for $n\ge3$},\,\phi^-_{i0}(s)=O(|\log s|)\,\text{for $n=2$}\qquad s\to 0^+,\\ &\partial_s\phi^-_{i0}=O(s^{1-n})\qquad s\to 0^+
\end{aligned}
\end{equation}
\end{proposition}

Due to the convention of having repeated eigenvalues indexed with different indexes, $\nu_{0j}=N-1$ for $j=1,\ldots,n$ and $\nu_{i0}=N-1$ for $i=1,\ldots,m$.

\begin{remark}
\label{rem-zeta}
We note that, by Proposition \ref{prop-fund-syst}, the Jacobi field $\zeta(y):=y\cdotp\nu_\Sigma(y)$ always fulfils $\zeta(y)=|y|^{-\frac{N-2}{2}+\Lambda_0}(1+o(1))$ as $|y|\to\infty$. This already follows from the strict minimality of $C_{m,n}$ if either $m+n\ge 9$ or $m+n=8$ with $|m-n|\le 2$, while the result is new in case $m=2$, $n=6$ or $m=6$, $n=2$. In particular, the hypersurfaces constructed in Theorem \ref{th_Al} fulfil $\bar{\delta}=\Lambda_0>0$, as we mentioned in the proof of Proposition \ref{Prop-injectivity-part-case}.
\end{remark}
\begin{corollary}
Equation (\ref{eq-Jacobi-modos}) for $i=0$ and $j=1,\ldots,n$, admits a bounded solution $\phi^+_{0j}$ such that $\phi^+_{0j}(s)\to 1$ as $s\to\infty$ and $\phi^+_{0j}(s)=O(s)$ as $s\to 0^+$.
\label{lemma-0,1}
\end{corollary}
\begin{corollary}
Equation (\ref{eq-Jacobi-modos}) for $j=0$ and $i=1,\ldots,m$ admits a bounded solution $\phi^+_{i0}$ such that $\phi^+_{i0}(s)\to 1$ as $s\to\infty$ and $\partial_s\phi^+_{i0}(s)\to 0$ as $s\to 0^+$.
\label{lemma-1,0}
\end{corollary}
Now we prove Proposition \ref{prop-fund-syst}.
\begin{proof}[Proof of Proposition \ref{prop-fund-syst}]

This proof is for all values of $i,j$.\\

First we consider the case $n\ge 3$. The asymptotic behaviour of $V_{ij}$ is given by
\begin{equation}\notag
V_{ij}(t)=-\left(\frac{n-2}{2}\right)^2-\mu_j^n+O(e^{2t})\qquad t\to-\infty
\end{equation}
so, by Lemma \ref{lemma-syst}, we can choose a fundamental system of solutions $w^\pm_{ij}$ to equation (\ref{eq-Jacobi-modos-EF}) such that
\begin{equation}
\label{as-be-w_0,1}
w^\pm_{ij}(t)\sim e^{\pm\sqrt{\left(\frac{n-2}{2}\right)^2+\mu_j^n}t},\,\partial_t w^\pm_{ij}(t)\sim\pm\sqrt{\left(\frac{n-2}{2}\right)^2+\mu_j^n}e^{\pm\sqrt{\left(\frac{n-2}{2}\right)^2+\mu_j^n}t}\text{ as $t\to-\infty$}
\end{equation}
Going back to the original variables, or in other words setting $\phi^+_{ij}(s):=p(t)w^+_{ij}(t)$ and multiplying by a constant again if necessary, we can see that $\phi^+_{ij}$ satisfy the required asymptotic behaviour as $s\to 0^+$. We note that, for $j>0$, only the asymptotic behaviour of $w^+_{ij}$ is used, while for $j=0$ we actually need the remainders estimates provided by Lemma \ref{lemma-syst} for the derivatives $\partial_s\phi^+_{ij}$.\\

In order to deal with the asymptotic behaviour of $\phi^+_{ij}$ as $s\to\infty$ we note that
$$V_{ij}(t)=-\left(\frac{N-2}{2}\right)^2+(N-1)-\nu_{ij}+O(e^{-\eta t})\qquad t\to\infty$$
so that, by Lemma \ref{lemma-syst} applied to the reflected equation $\partial^2_t w+V_{ij}(-t)w=0$, we can prove the existence of a fundamental system $v^\pm_{ij}$ whose asymptotic behaviour at infinity is given by
$$v^\pm_{ij}\sim e^{\pm\Lambda_{ij}t},\,\partial_t v^\pm_{ij}\sim e^{\pm\Lambda_{ij}t}=\pm\Lambda_{ij} e^{\pm\Lambda_{ij}t}\qquad t\to\infty,$$
where $\Lambda_{ij}^2:=\left(\frac{N-2}{2}\right)^2-(N-1)+\nu_{ij}$. Writing
$$\phi^+_{ij}(s)=(\alpha^+_{ij} v^+_{ij}(t)+\alpha^-_{ij}v^-_{ij}(t))p(t)$$

We claim that $\alpha_{ij}^+\neq 0$. If $\alpha^+_{ij}=0$, then the corresponding Jacobi field $$\psi(y):=\phi^+_{ij}(s)u_i(\x)v_j(\y),\qquad\forall\,y=(a(s)\x,b(s)\y)\in \Sigma,$$ 
would belong to $X$ (see \eqref{def-X}), 
which is not possible, due to the strict stability of $\Sigma$ (see Proposition \ref{prop-strict-stability}) and this proves the claim. Therefore $\phi^+_{ij}$ has the required asymptotic behaviour as $s\to\infty$.\\ 

Proceeding in a similar manner, we can see that $\phi^-_{ij}(s)=v^-_{ij}(t)p(t)$ has the required behaviour as $s\to\infty$. On the other hand, using the fundamental system $w^\pm_{ij}$, we can write
$$\phi^-_{ij}(s)=(\beta^+_{ij}w^+_{ij}(t)+\beta^-_{ij} w^-_{ij}(t))p(t).$$
We claim now that $\beta_{ij}^-\neq 0$. If $\beta^-_{ij}=0$, then the corresponding Jacobi field 
$$\tilde{\psi}(y):=\phi^+_{ij}(s)u_i(\x)v_j(\y),\qquad\forall\,y=(a(s)\x,b(s)\y)\in \Sigma,$$ 
would belong to $X$, which as before yields a contradiction. As a consequence, $\phi^-_{ij}$ fulfils the required asymptotic behaviour as $s\to 0^+$ to.\\

The case $n=2$ is similar, it relies on part $(2)$ of Lemma \ref{lemma-syst} to compute the asymptotic behaviour of $\phi_{i0}^\pm$.
\end{proof}
Now we are ready to prove Theorem \ref{th-Jac-Sigma}.
\begin{proof}[Proof of Theorem \ref{th-Jac-Sigma}]
First we note that any bounded Jacobi field $\phi$ of $\Sigma$ is $C^\infty(\Sigma)$. In fact $\phi\in L^\infty(\Sigma)$ implies that $\phi\in L^p_{loc}(\Sigma)$ for any $p>1$, so that, by the elliptic estimates, $\phi\in W^{2,p}_{loc}(\Sigma)$ for any $p>1$. Taking $p>N$ we can see that $\phi\in C^{1,\alpha}_{loc}(\Sigma)$ with $\alpha:=1-\frac{N}{p}$ so that, by a bootstrap argument, $\phi$ is smooth.\\

In order to prove that dim $K(\Sigma)=N+2$ we first observe that the symmetric Jacobi field $y\mapsto y\cdotp \nu_\Sigma(y)=\phi_{00}^+(s)\in K(\Sigma)$, moreover the only bounded solutions $\phi_{ij}^\pm$ of equations (\ref{eq-Jacobi-modos}) are $\phi^+_{00}$, $\phi_{0j}^+$ for $j=1,\ldots,n$ and $\phi^+_{i0}$ for $i=1,\ldots,m$.\\

The eigenspace corresponding to $\mu^m_{1}=\cdots=\mu^m_{m}$ on $S^{m-1}$ has dimension $m$. Taking a basis $\{u_{i}\}_{1\le i\le m}$ of such space we obtain $m$ linearly independent bounded Jacobi fields $\Phi_i(y):=\phi^+_{i0}(s)u_{i}(\x)$ of $\Sigma$ which are not $O(m)\times O(n)$-invariant. Similarly, $\Psi_j(y):=\phi^+_{0j}(s)v_{j}(\y)$ are $n$ linearly independent Jacobi fields of $\Sigma$ enjoying the same properties.\\

Summarizing, there is one linearly independent symmetric bounded Jacobi field (generated by dilations of $\Sigma$) and $N+1$ non-symmetric bounded linearly independent Jacobi fields.\\ 

Appealing to Remark \ref{rmk:space_transl} we define
$$
T={\rm span}\{\Phi_i\}_{i=1}^m \oplus {\rm span}\{ \Psi_j\}_{j=1}^n,
$$
so that $K(\Sigma)=T\oplus {\rm span}\{y\cdotp\nu_\Sigma(y)\}$.\\ 

Since the Jacobi fields generated by translations $y\mapsto \nu_{\Sigma}(y)\cdot a$  belong to $L^{\infty}(\Sigma)\backslash{\rm span}\{y\cdotp\nu_\Sigma(y)\}$, for any $a\in\R^{N+1}\backslash\{0\}$, we can use their Fourier decomposition to prove that the space $$\tilde{T}:=\{\nu_{\Sigma}(y)\cdot a:\,a\in\R^{N+1}\}$$
is a subspace of $T$. Since $\dim(T)=\dim(\tilde{T})=N+1$, then  $T=\tilde{T}$. This concludes the proof.
\end{proof}

Now we study the Jacobi fields of the surfaces constructed in Theorem \ref{th-Sigma-low-dim}. In particular we prove the following result
\begin{theorem}
\label{th-Jabobi-Sigma-low-dim}
Let $\Sigma$ be one of the surfaces constructed in Theorem \ref{th-Sigma-low-dim}. Then the space of bounded $O(m)\times O(n)$-invariant Jacobi fields is generated by the Jacobi field $\zeta(y):=y\cdotp \nu_\Sigma(y)$, corresponding to dilations. Moreover $|\zeta(y)|\le c|y|^{-\frac{N-2}{2}}$ and $$\limsup_{|y|\to\infty}|y|^{\frac{N-2}{2}}|\zeta(y)|> 0.$$ 
\end{theorem}
\begin{remark}
Theorem \ref{th-Jabobi-Sigma-low-dim} implies Theorem \ref{th-nondegeneracy}, part $(2)$.
\end{remark}
The proof relies on the following Lemma.
\begin{lemma}
\label{lemma-syst-fund-inst}
Let $V\in C^\infty(\R)$ be a potential such that
\begin{equation}
\label{cond-V-inst}
|V(t)-\lambda^2|\le ce^{\eta t} \qquad\text{for $t\in(-\infty,t_0)$,}
\end{equation}
for some $\lambda,\eta>0$, $t_0\in\R$. Then the equation
\begin{equation}
\label{eq-Jacobi-EF-inst}
\partial_{tt}w+V(t)w=0
\end{equation}
has two linearly independent solutions $w^\pm$ such that $w^+\sim\cos(\lambda t)$, $w^-\sim\sin(\lambda t)$ as $t\to-\infty$ and their derivatives are bounded in $(-\infty,t_0)$
\end{lemma}
\begin{proof}
We look for a solution $w$ to equation (\ref{eq-Jacobi-EF-inst}) of the form $w=e^{i\lambda t}x$. Differentiating twice it is possible to see that equation (\ref{eq-Jacobi-EF-inst}) is equivalent to the system
\begin{equation}
\begin{cases}
x_t=e^{-2i\lambda t}y\\
y_t=-q(t)e^{2i\lambda t}x,
\end{cases}
\end{equation}
where we have set $V(t)=\lambda^2+q(t)$. Imposing the initial conditions $x(t_0)=1$ and $y(t_0)=0$, the fundamental Theorem of calculus gives
$$x(t)=1-\int_t^{t_0} e^{-2i\lambda \zeta}y(\zeta)d\zeta,\qquad y(t)=\int_t^{t_0}q(\tau)e^{2i\lambda \tau}x(\tau)d\tau,$$
for $t<t_0$. By Fubini's Theorem we have
$$x(t)=1+i\int_t^{t_0}x(\tau)q(\tau)\frac{1-e^{2i\lambda(\tau-t)}}{2\lambda}d\tau,$$
so that
$$|x(t)|\le 1+\frac{1}{\lambda}\int_t^{t_0}|x(\tau)||q(\tau)| d\tau.$$
By Gronwall's inequality, this yields that
$$|x(t)|\le \exp\left(\int_t^{t_0}|q(\tau)|d\tau\right)\le \exp\left(\int_{-\infty}^{t_0}|q(\tau)|d\tau\right)<\infty,$$
so that $x\in L^\infty(-\infty,t_0)$. Taking $\tilde{t}<t<t_0$, we have
$$|x(t)-x(\tilde{t})|\le\left|\int_{\tilde{t}}^t e^{-2i\lambda\zeta}d\zeta\int_{t_0}^\zeta q(\tau)e^{2i\lambda\tau}d\tau\right|\le\int_{\tilde{t}}^t d\zeta\int_{t_0}^\zeta |q(\tau)|d\tau\le ce^{\eta t}\to 0$$
as $t\to\infty$, so that there exists $\underline{x}:=\lim_{t\to-\infty}x(t)$. Due to the initial condition, the real part of $\underline{x}$ is positive, so that $\underline{x}\ne 0$, so that $w(t)\sim \underline{x} e^{i\lambda t}$ as $t\to-\infty$. Moreover, it is possible to prove the lower bound
$$|y(t)|\le c\int_t^{t_0}|q(\tau)|d\tau\le c\int_{-\infty}^{t_0} e^{\eta \tau} d\tau<\infty,$$
so that $y\in L^\infty(-\infty,t_0)$ is also bounded. As a consequence the derivative $x_t=e^{i\lambda t}(x_t+i\lambda x)$ is also bounded in $(-\infty,t_0)$. As a consequence, taking a suitable linear combination of the real and the imaginary part of $w$, we have the statement.
\end{proof}
\begin{remark}
Lemma \ref{lemma-syst-fund-inst} actually shows that any non-trivial solution to equation (\ref{eq-Jacobi-EF-inst}) is bounded in a neighbourhood of $-\infty$ but it does not decay at $-\infty$. 
\end{remark}
Now we can prove Theorem \ref{th-Jabobi-Sigma-low-dim}
\begin{proof}
Any $O(m)\times O(n)$-invariant Jacobi field $\psi$ can be written as $\phi(y)=\varphi(s)$, where $y=(a(s)\x,b(s)\y)$ and $\varphi$ fulfils the ODE
$$\partial_{ss}\varphi+\alpha(s)\partial_s\varphi+\beta\varphi=0$$
in $(0,\infty)$, hence the space of $O(m)\times O(n)$-invariant Jacobi fields has dimension $2$. Introducing an Emden-Fowler change of variables $s=e^t$ and writing $\varphi(s)=p(t)w(t)$, we can see that $w$ fulfils (\ref{eq-Jacobi-EF-inst}) with $V:=V_{0,0}$. Using the asymptotic behaviour of $V$ at $\pm\infty$, namely
$$V(t)=
\begin{cases}
-\left(\frac{n-2}{2}\right)^2+O(e^{2t})\qquad t\to-\infty\\
-\left(\frac{N-2}{2}\right)^2+(N-1)+O(e^{\eta t})\qquad t\to-\infty\\
\end{cases}
$$
Lemma \ref{lemma-syst-fund-inst} and Lemma (\ref{lemma-syst}) we can see that (\ref{eq-Jacobi-EF-inst}) admits a fundamental system $w^\pm$ such that $w^+$ is bounded and $w^-$ is bounded in $(0,\infty)$ but explodes as $t\to-\infty$. More precisely, it fulfils
$$w^-(t)=
\begin{cases}
O(e^{-\frac{n-2}{2}t})\qquad\text{for $n>2$} \\
O(t)\qquad\text{for $n=2$}
\end{cases}
$$
as $t\to-\infty$. As a consequence, writing
$$\phi(y)=p(t)(\alpha ^+w^+(t)+\alpha^-w^-(t))$$
and using that $\phi$ is bounded, we have $\alpha^-=0$, which shows that the space $K(\Sigma,O(m)\times O(n))$ of bounded $O(m)\times O(n)$-invariant Jacobi fields has dimension $1$.\\ 

Since $\zeta(y):=y\cdotp \nu_\Sigma(y)\in K(\Sigma,O(m)\times O(n))\backslash\{0\}$, it generates the space.\\

Due to Theorem \ref{th-dec-w-low-dim}, $\zeta$ is sign-changing $|\zeta(y)|\le c|y|^{-\frac{N-2}{2}}$, the decay rate being sharp. This vanishing property is consistent with the instability of our hypersurfaces.
\end{proof}
\begin{remark}
    Propositions $3$ and $5$ of \cite{M} actually show that, if $\Sigma$ is one of the hypersurfaces constructed in Theorem \ref{th-Sigma-low-dim}, the Jacobi field $\zeta(y):=y\cdotp\nu_\Sigma(y)$ vanishes infinitely many times and the decay rate at infinity given by Theorem \ref{th-Jabobi-Sigma-low-dim} is sharp.
\end{remark}

\end{document}